\documentclass[a4paper]{amsart}

\usepackage{amssymb}
\usepackage{amsthm}  
\usepackage{amsmath} 
\usepackage{amscd} 
\usepackage{float}
\usepackage[all]{xypic}
\usepackage{url}
\usepackage{hyperref}
\usepackage{color}
\hypersetup{linktocpage}

\newcommand{\fg}{\mathfrak g}

\newcommand{\fr}{\mathfrak r}

\newcommand{\fn}{\mathfrak n}

\newcommand{\fm}{\mathfrak m}

\newcommand{\Ad}{{\rm Ad}}

\numberwithin{table}{section}

\theoremstyle{plain}

\theoremstyle{plain}
\newtheorem{thm*}{Theorem}[section]
\newtheorem{prop*}[thm*]{Proposition}
\newtheorem{lem*}[thm*]{Lemma}
\newtheorem{cor*}[thm*]{Corollary}

\theoremstyle{definition}
\newtheorem{def*}{Definition}

\theoremstyle{remark}
\newtheorem{rem*}{Remark}

\begin{document}
\title[On the Beloshapka's rigidity conjecture]{On the Beloshapka's rigidity conjecture for real submanifolds in complex space}
\author{Jan Gregorovic}
\address{Faculty of Mathematics, University of Vienna, Oskar Morgenstern Platz 1, 1090 Wien, Austria}
 \email{jan.gregorovic@seznam.cz}
\subjclass[2010]{}
 \thanks{The author was supported by the project P29468 of the Austrian Science Fund (FWF). The author would like to thank I.G. Kossovskiy for the introduction and the discussion about the topic of the article.}
\maketitle
\begin{abstract} 
A well known Conjecture due to Beloshapka asserts that all totally nondegenerate polynomial models with the length $l\geq 3$ of their Levi-Tanaka algebra are {\em rigid}, that is, any point preserving automorphism of them is completely determined by the restriction of its differential at the fixed point onto the complex tangent space. For the length $l=3$, Beloshapka's Conjecture was proved by Gammel and Kossovskiy in 2006. In this paper, we prove the Conjecture for arbitrary length $l\geq 3$. 

As another application of our method, we construct polynomial models of length $l\geq 3$, which are not totally nondegenerate and admit large groups of point preserving nonlinear automorphisms. 
\end{abstract}

\section{Introduction}

 The concept  of a {\em polynomial model} for real submanifolds in complex space goes back to the classical work of Poincar\'e \cite{poincare}, Cartan \cite{cartan}, Tanaka \cite{T}, Chern and Moser \cite{chern}. In the cited work, it was demonstrated that the key objects associated with the real submanifolds in complex space (such as the automorphism groups, normal forms, canonical connections, etc) can be studied by employing appropriate polynomial models. We emphasize that the work of Poincar\'e-Cartan-Tanaka-Chern-Moser mainly concerned submanifolds of codimension one (real hypersurfaces) satisfying the {\em Levi-nondegeneracy} condition. (Tanaka considered also the cases $k=n^2,\,k=n^2-1$, where $n$ is the CR--dimension and $k$ is the CR--codimension of a generic real submanifold).

An important development towards understanding the concept of a polynomial model for general real submanifolds in complex space was done in the work of Bloom-Graham \cite{BL}. Bloom-Graham consider a real submanifold $M\subset\mathbb C^{n+k}$ (without restrictions on $n,k$ besides their positivity) satisfying the Hormander-Kohn bracket condition: Lie brackets of the local sections of the complex tangent bundle $T^{\mathbb{C}}M:=TM\cap i(TM)$ generate the entire tangent bundle $TM$. With any such $M$, Bloom-Graham associate a sequence of integers $(m_1,\dots,m_k)$ that they call {\em type} (at a point $p\in M$), and further associate a polynomial submanifold $M_0\subset\mathbb C^{n+k}$ given by a system of equations
\begin{equation}\label{bg}
\mbox{Im}(w_1)=P_1(z,\bar z,\mbox{Re}(w)),\cdots, \mbox{Im}(w_k)=P_k(z,\bar z,\mbox{Re}(w)),\quad z\in\mathbb C^n,\,\,w\in\mathbb C^k,
\end{equation} 
where each $P_j$ is a weighted homogeneous polynomial, if one uses appropriate integral weights
$$[z_j]=1,\quad [w_j]=m_j,\quad m_j\geq 2.$$
In what follows, we call the latter submanifolds {\em weighted homogeneous models}.

Systematic study of the Bloom-Graham weighted homogeneous models {\em in the real-analytic case} was initiated  by the school of Vitushkin (see e.g. the surveys \cite{vitushkin}, \cite{obzor}), particularly in the work of Beloshapka. In his pioneering work \cite{belold}, Beloshapka introduced the concept of a {\em nondegenerate quadric} model, which extends the Poincar\'e-Cartan-Tanaka-Chern-Moser concept of a Levi--nondegenerate polynomial model to the case of high codimension ($2\leq k \leq n^2$). In this case, all the above $P_j$ are linearly independent Hermitian forms in $z$ without common kernel. Beloshapka proves the finite-dimensionality of the automorphism group of a nondegenerate quadric models, and as well demonstrates the fundamental bound:
$$\mbox{dim\,Aut}\,(M,p)\leq  \mbox{dim\,Aut}\,(Q,0),$$
where $M$ is a Levi-nondegenerate submanifold with a distinguished point $p\in M$ and $Q$ is its associated quadratic model at $p$. By employing the quadratic models, Beloshapka as well demonstrates {\em the $2$-jet determination} of CR maps between Levi-nondegenerate real submanifolds. 

The cited work of Beloshapka inspired a big interest to studying nondegenerate quadric models. A lot of work since 1990's till present has been dedicated to the classification of quadric models,  studying
their symmetries (in particular, looking for quadrics with {\em nonlinear} symmetries), linearization problems for symmetry groups of Levi-nondegenerate submanifolds, normal forms for Levi-nondegenerate submanifolds. An incomplete list of authors here includes Ezhov and Schmalz (e.g. \cite{ES1,ES2,ES3}), Lamel and Stolovitch \cite{LS17},   Garrity-Mizner \cite{GM,Mi}, Isaev (e.g. \cite{Is}), Isaev-Kaup \cite{IK}, Shevchenko \cite{shev1,shev2}.

At the same time, important developments of the work \cite{belold} for the jet determination problem for mappings between very general real submanifolds in complex space were initiated by the school of Baouendi-Rothschild. Without an ambition to give a complete list of publications here, we mention Baouendi-Ebenfelt-Rothschild \cite{ber0,ber}, Zaitsev \cite{zai1,zai2} (see also references therein), Ebenfelt-Lamel-Zaitsev \cite{elz}.   See also Kolar \cite{kolar} for exact jet determination bounds  in the $\mathbb C^2$ case.

It as well worth mentioning that the geometry of the Beloshapka's nondegenerate quadric models is closely related to that of Bounded Symmetric Domains: see e.g. Kaup-Matsushima-Ociai \cite{km}, Tumanov \cite{tumanov}, Tumanov-Henkin \cite{th}.

We shall emphasize, however, that quadrics are applicable as nondegenerate models only in the range of CR--codimensions
$$1\leq k\leq n^2.$$ That is why in the case $k>n^2$ one needs to consider nondegenerate models with polynomials $P_j$ of higher degree (so-called models of super-high codimension). To study the latter case, Beloshapka introduced in a series of publications (where the final step was made in the paper \cite{B}) the concept of a {\em totally nondegenerate model}. Total nondegeneracy is a (generic!) property of the Bloom-Graham type $(m_1,\dots,m_k)$ of a real submanifold, which is explained in detail in the next section (see Definition \ref{def2}), while an informal definition is as follows: there are no linear dependences between Lie brackets of local sections of  $T^{\mathbb{C}}M$ besides that forced by the axioms within the Lie algebra of germs of complex vector fields on $M$. Then, for any $n,k>0$ and any totally nondegenerate model with given $n,k$, Beloshapka constructs a family of weighted homogeneous models. Germs of these models approximate in a natural way germs of  arbitrary totally nondegenerate submanifolds, and one has the bound
 $$\mbox{dim\,Aut}\,(M,p)\leq  \mbox{dim\,Aut}\,(M_0,0),$$
where $M$ is a Levi-nondegenerate submanifold with a distinguished point $p$ and $M_0$ is its associated model at $p$. (The finite dimensionality of  $\mbox{Aut}\,(M_0,0)$ can be verified also from e.g. \cite{ber}, though Beloshapka provides in \cite{B} more precise bounds specific for the totally nondegenerate case). An important invariant of the totally nondegenerate models is their {\em length} $l$. The latter can be viewed, on one hand, as the maximal possible depth of a Lie bracket in the definition of the Bloom-Graham type, and at the same time, $l$ is the maximal possible weighted degree of a polynomial   in \eqref{bg}. The length $l$ is precisely the length of the {\em Levi-Tanaka algebra} (see the next section). In terms of length, Beloshapka's totally nondegenerate models look at a glance as 
\begin{gather*}
\mbox{Im}(w_1)=P^1_{2}(z,\bar z),...,\mbox{Im}(w_{n^2})=P^{n^2}_{2}(z,\bar z),\\ 
...\\
\mbox{Im}(w_{k-s+1}=P^1_{l}(z,\bar z,\mbox{Re}(w)),...,\mbox{Im}(w_k)=P^s_{l}(z,\bar z,\mbox{Re}(w)),\\ 
z\in\mathbb C^n,\,\,w\in\mathbb C^k
\end{gather*}
(here each $P^a_{b}$ has weighted degree $b$). For more details we refer, again, to the next section.

Beloshapka's totally nondegenerate models have been studied  intensively in the last 20 years (e.g. Beloshapka-Ezhov-Schmalz \cite{bes}, Gammel-Kossovskiy \cite{G}, Kossovskiy \cite{4degree}, Shananina (\cite{shan1,shan2}), Mamai \cite{mamai}, Sabzevari \cite{sabzevari}). See also in this regards the survey of Kolar-Kossovskiy-Zaitsev \cite{kkz}. Somewhat surprisingly, despite of all efforts, one could never construct a totally nondegenerate model with nonlinear origin preserving automorphisms.  It has been a long standing problem whether such models exist, and this led to

\smallskip

\noindent{ \bf Beloshapka's Rigidity Conjecture (\cite{B},\cite{obzor}).} All totally nondegenerate models of the length $l\geq 3$ are {\em rigid}, i.e., admit no nonlinear origin preserving automorphisms.

\smallskip

Alternatively, the Beloshapka's Conjecture can be formulated as the nonexistence of totally nondegenerate models with length $l\geq 3$ admitting infinitesimal CR automorphisms that have positive weighted degree. We note that for $l=2$ (the case of quadrics) such nonlinear automorphisms often exist (see in e.g. \cite{chern} for the case of hyperquadrics, or \cite{obzor} for numerous examples in higher codimension). 

\medskip

For the length $l=3$ the Beloshapka's Conjecture was proved by Gammel and Kossovskiy in \cite{G}. It was as well proved by Kossovskiy for a subclass of models of length $4$ in \cite{4degree}. 

\smallskip

We also mention that for the special CR--dimension $n=1$ rigidity holds automatically as well, as was observed by Kruglikov \cite{krug} and can be easily deduced from results in \cite{M}. 

\smallskip

The mail goal of this paper is to solve Beloshapka's Conjecture in its full generality in the affirmative.

\smallskip

\noindent{\bf Main Theorem.} {\em Any totally nondegenerate model $M_0$ of length $l\geq 3$ and arbitrary CR--dimension $n>0$ and CR--codimension $k>0$ is rigid, i.e. the stability group of the origin contains only linear automorphisms (which are then determined by their restrictions on the complex tangent space at $0$ to $M_0$).}

\smallskip

\noindent{\bf Corollary.} {\em Let $M,N$ be two real-analytic totally nondegenerate real submanifolds in complex space with a distinguished point $p,q$. Let $M,N$ have length $l\geq 3$ at $p,q$, respectively. Then any local automorphism $H: M\to N$ near $p$ with $H(p)=q$ is uniquely determined by the restrictions of its differential $T_pH$ onto the complex tangent space $T_p^{\mathbb{C}}M$.}

\smallskip

We finally mention that another well known question of Beloshapka asks about the existence of {\em non-rigid} and possibly not totally nondegenerate polynomial models of length $l\geq 3$ with finite dimensional automorphism groups. We give in this paper an answer to the latter question as well, by constructing (not totally nondegnerate) polynomial models with large nonlinear stability groups of the origin. Interestingly, these examples exhibit {\em parabolic geometries} (see e.g. \cite{P}) within the class of weighted homogeneous models.

{\em Acknowledgement}. After posting the article on arXiv, I noticed that there is a paper \cite{sabzevari2} by Sabzevari and Spiro that appeared practically at the same time that is independently proving the Main Theorem by slightly different methods.

\section{Preliminaries and detailed statement of results}\label{prel}

\subsection{Preliminaries}

The Beloshapka's totally nondegenerate models are not generic weighted homogeneous models, e.g., in \cite{B}, it is shown that Beloshapka's totally nondegenerate models are always homogeneous w.r.t. the CR automorphisms group. Therefore we will consider only a particular class of weighted homogeneous models. However, it is not convenient for us to restrict ourselves to a class of weighted homogeneous models in a particular normal form suggested for the weighted homogeneous models in \cite{BL} or for the totally nondegenerate models in \cite{B}. The results and terminology in \cite{B} suggest us to use the following class of weighted homogeneous models.

\begin{def*}
We say that a CR submanifold $M_0\subset \mathbb{C}^{n+k}$ given by a system of equations (\ref{bg}) is a {\em model surface} of codimension $k$ if the following conditions are satisfied:
\begin{enumerate}
\item The functions $P_j(z,\bar z,\mbox{Re}(w))$ are polynomials and after the assignment of weight $[z_i]=1$ to all $z$ variables there is a unique assignment of integral weights $[w_i]=m_i,\ m_i\geq 2$ to each $w$ variable under which all polynomials $P_j(z,\bar z,\mbox{Re}(w))$ are weighted homogeneous.
\item The Lie brackets of the local sections of the complex tangent bundle $T^{\mathbb{C}}M_0$ generate the entire tangent bundle $TM_0$.
\item The submanifold $M_0$ is homogeneous w.r.t. the CR automorphism group.
\item The submanifold $M_0$ is Levi--nondegenerate at $0$.
\end{enumerate}

We say that that the sequence $(m_1,\dots,m_k)$ is the {\em Bloom--Graham type} of $M_0$. 
\end{def*}

\begin{def*}\label{def2}
We say that a model surface $M_0\subset \mathbb{C}^{n+k}$ of codimension $k$ is {\em totally nondegenerate} if its Bloom--Graham type is (lexicographically) minimal among all model surfaces of $\mathbb{C}^{n+k}$ of the same codimension $k$.
\end{def*}

According to \cite[Theorem 6.2]{BL}, each model surface $M_0$ can be transformed by weighed degree preserving holomorphic transformation into a model surface in a normal (standard) form. It is a simple observation that the Beloshapka's totally nondegenerate models from \cite{BL} are exactly the normal forms of the totally nondegenerate model surfaces.

The homogeneity of the model surface $M_0$ w.r.t. the CR automorphism group implies that each point $p\in M_0$ is {\em regular}, i.e., all subbundles $T^{-i}M_0$ of $TM_0$ generated by $i$ sections of $T^{\mathbb C}M_0$ have constant rank. At a regular point $p$ of a CR submanifold $M$, Tanaka defines in \cite{T} a {\em Levi--Tanaka algebra} $(\fm,I)$ of $l$--th type of $M$ at $p$ that is 
\begin{enumerate}
\item a graded Lie algebra $$\fm=\fg_{-l}\oplus \dots \oplus \fg_{-1}$$ such that $\fm$ is generated by $\fg_{-1}$, the Lie bracket $\fg_{-1}\otimes \fg_{-1}\to \fg_{-2}$ is non--degenerate and $\fg_{-l}\neq 0$, and \item a complex structure $I$ on $\fg_{-1}$ such that $$[I(X),I(Y)]=[X,Y]$$ holds for all $X,Y\in \fg_{-1}.$
\end{enumerate} We recall the full details on the construction of  the Levi--Tanaka algebra in Section \ref{Sec21}, however, now we only mention that the vector spaces $\fg_{-i}$ are defined as the quotients $T^{-i}_pM_0/T^{-i+1}_pM_0$. Therefore $dim(\fg_{-i})$ for $i>1$ is the number of $i$ in the Bloom--Graham type of $M_0$ at $p$ and the total nondegeneracy means that $dim(\fg_{-i})$ for $i<l$ is the maximal possible. It is well know from \cite{T} and the theory of Lie algebras, which Levi--Tanaka algebras have this property, see Section \ref{Sec22}. This immediately provides the informal characterization of the total nondegeneracy from the Introduction.

A pair $(\fm,I)$ with the properties (1) and (2) of the Levi--Tanaka algebra is called a {\em non--degenerate pseudo--complex fundamental graded Lie algebra} $(\fm,I)$ of $l$--th type. Tanaka shows in \cite{T} that for each non--degenerate pseudo--complex fundamental graded Lie algebra $(\fm,I)$ there is a flat model that has Levi--Tanaka algebra $(\fm,I)$ at each point. The flat model is the nilpotent Lie group $\exp(\fm)$ together with the left invariant distribution and the complex structure given by $\fg_{-1}$ and $I$. In \cite{N}, Naruki shows that, for the flat model $\exp(\fm)$, there is an CR embedding $$\phi:\fm\to \mathbb{C}^N$$ for $N=dim_\mathbb{R}(\fm)-dim_\mathbb{C}(\fg_{-1})$. The CR submanifold $\phi(\fm)\subset \mathbb{C}^N$ is called a {\em standard real submanifold} corresponding to the non--degenerate pseudo--complex fundamental graded Lie algebra $(\fm,I)$. In Section \ref{Sec23}, we show that each standard real submanifold $\phi(\fm)\subset \mathbb{C}^N$ is in fact a model surface, i.e., each non--degenerate pseudo--complex fundamental graded Lie algebra is a Levi Tanaka algebra at $0$ of a model surface. One of the results of this article, that we present in the next section is that the converse statement is also true, each model surface $M_0$ is isomorphic to the standard real submanifold $\phi(\fm)\subset \mathbb{C}^N$ corresponding to the Levi--Tanaka algebra $(\fm,I)$ of $M_0$ at $0$.

\subsection{Detailed statement of results}

Our first result (Theorem \ref{Thm1}) relates the standard real submanifolds and the model surfaces. In particular, it provides description of the totally nondegenerate model surfaces using the universal Levi--Tanaka algebras of $l$--th type, see Section \ref{Sec22}. Moreover, it shows that the standard real submanifold $\phi(\fm)\subset \mathbb{C}^N$ corresponding to Levi--Tanaka algebra $(\fm,I)$ of $M_0$ at $0$ provides a normal form for the model surface $M_0$ (different than \cite{BL}), see Section \ref{Sec23} for the details. Finally, it allows to explicitly compute the infinitesimal CR automorphims of the model surfaces in the coordinates of this normal form, see Section \ref{Sec24}.

\begin{thm*}\label{Thm1}
Let $M$ be a CR submanifold in $\mathbb{C}^{n+k}$ of CR--dimension $n$ containing $0$. Then the following conditions are equivalent:
\begin{enumerate}
\item $M$ is isomorphic to a model surface $M_0$ of codimension $k$.
\item $M$ is isomorphic to the standard real submanifold $\phi(\fm)\subset \mathbb{C}^{n+k}$ corresponding to the Levi--Tanaka algebra $(\fm,I)$ of $l$--th type of $M$ at $0$.
\end{enumerate}

Moreover, if one of the conditions (1) and (2) holds, then $M_0$ is a totally nondegenerate model surface if and only if the Levi--Tanaka algebra of $M$ at $0$ is a quotient of the universal Levi--Tanaka algebra of $l$--th type by a linear subspace in $-l$--th grading component.
\end{thm*}

We prove the Theorem \ref{Thm1} in Section \ref{Sec3}.

The Theorem \ref{Thm1} and results from \cite{T} reduce the computation of the nonlinear CR automorphims of the model surfaces $M_0$ to the computation of the Tanaka prolongation $\fg$ of the Levi--Tanaka algebra $(\fm,I)$ of $M_0$ at $0$. Moreover, we show in Section \ref{Sec24}, that on the standard real submanifold $\phi(\fm)\subset \mathbb{C}^N$, there are the infinitesimal CR automorphisms corresponding to elements of $\fg_i\subset \fg$ have weighted degree $i$ (as holomorphic vector fields on $\mathbb{C}^N$). Thus the positive part $\fg_+= \fg_{1}\oplus \dots \oplus \fg_l$ of the Tanaka prolongation $\fg$ corresponds to the space of infinitesimal CR automorphisms of positive weighted degree, i.e., to the nonlinear CR automorphisms in $Aut(M_0,0)$. Therefore we can rephrase the Main Theorem (the Beloshapka's conjecture) as:

\begin{thm*}\label{Thm3}
It holds $\fg_+=0$ for the Tanaka prolongation of the Levi--Tanaka algebra at $0$ of the totally nondegenerate model surface $M_0$ of codimension $k$ and length $l\geq 3$.
\end{thm*}

Due to the results indicated in the Introduction, it suffices to resolve the Beloshapka's conjecture for $l>3$. In order to do this, we prove in the Section \ref{Sec4} the following result that holds for a larger class of model surfaces.

\begin{thm*}\label{Thm2}
Let $M_0$ be a model surface in $\mathbb{C}^{n+k}$ of codimesion $k$ with length $l>3$. If one of the following equivalent conditions is satisfied for some $K$ such that $l+2\leq 2K$:
\begin{enumerate}
\item the part $(m_1,\dots,m_c)$ of Bloom--Graham type $(m_1,\dots,m_k)$ corresponding to all variables $w$ of degree less of equal $K$ is (lexicographically) minimal among all model surfaces in $\mathbb{C}^{n+k}$ of codimension $k$, or
\item  the dimensions of spaces $\fg_{-2}, \dots,\fg_{-K}$ in the Levi--Tanaka algebra of $l$--th type of $M_0$ at $0$ are the same as the dimensions of the corresponding grading components of the universal Levi--Tanaka algebra.
\end{enumerate}
then it holds $\fg_+=0$ for the Tanaka prolongation $\fg$ of Levi--Tanaka algebra of $M_0$ at $0$.
\end{thm*}

Indeed, each totally nondegenerate model surface of codimension $k$ with length $l>3$ satisfies the first condition of Theorem \ref{Thm2}. This finished the proof of the Main Theorem. As a part of the proof of Theorem \ref{Thm2} in the Section \ref{Sec4}, we need the classification of the standard real submanifolds with simple complex Lie algebra of infinitesimal CR automorphisms. In Section \ref{Sec5}, we recall the details on this classification using the results from \cite{A}. This classification is related to certain $|l|$--gradings of simple Lie algebras. Therefore this classification and results in Section \ref{Sec24} provide the following answer to the question of Beloshapka we recalled in the Introduction:

\begin{thm*}\label{Thm3}
For each $l\geq 2$, there is a model surface (not totally nondegenerate) that admits infinitesimal CR automorphisms of positive weighted degree. In particular, there are model surfaces that are (flat models of) parabolic geometries and carry infinitesimal CR automorphism of weighted degree $l$.
\end{thm*}

In Section \ref{Sec51}, we construct an explicit example of a weighted homogeneous model of length $l=3$ and large nonlinear stability group (with infinitesimal CR automorphisms of weighted degree up to $3$). In Appendix \ref{Sec6}, we provide several more interesting examples of  weighted homogeneous models that have large nonlinear stability group.

\section{Levi--Tanaka algebras and standard real submanifolds}\label{Sec2}

\subsection{CR manifolds and Levi--Tanaka algebras}\label{Sec21}

A CR manifold is a triple $(M,\mathcal{D},I)$, where $M$ is a smooth manifold with a smooth maximally non--integrable distribution $\mathcal{D}$ carrying a complex structure $I$ that satisfies a (formal) integrability condition. If we decompose the complexification $\mathcal{D}\otimes \mathbb{C}$ of $\mathcal{D}$ to $\pm i$--eigenspaces $\mathcal{D}^{10}\oplus \mathcal{D}^{01}$ of the complex structure $I$, then the (formal) integrability condition can be expressed as $[\mathcal{D}^{10},\mathcal{D}^{10}]\subset \mathcal{D}^{10}$. For a CR submanifold $M$ satisfying the Hormander-Kohn bracket condition is $\mathcal{D}$ just the complex tangent space $T^{\mathbb{C}}M$. On the other hand, not every CR manifold can be embedded in some $\mathbb{C}^N$.

Consider a CR manifold $(M,\mathcal{D},I)$ and let $p\in M$ be a regular point. Then the bracket of vector fields on $\mathcal{D}$ induces a tensorial (Levi) bracket in $\wedge^2\mathcal{D}_p^*\otimes T_pM/\mathcal{D}_p$. The CR manifold $(M,\mathcal{D},I)$ is called Levi non--degenerate at $p$ if this (Levi) bracket has trivial kernel. Thus at regular point $p$, the image of the Levi bracket coincides with the space $\fg_{-2}:=T^{-2}_pM/\mathcal{D}_p$ generated by Lie brackets of two local sections of $\mathcal{D}$. Iterating this procedure, we obtain tensorial brackets in $\mathcal{D}_p^*\otimes (T^{-i}_pM)^*\otimes T_pM/T^{-i}_pM$ with image $\fg_{-i+1}:=T^{-i-1}_pM/T^{-i}_pM$ at regular point $p$, where $T^{-i}M$ is the space generated by Lie brackets of $i$ local sections of $\mathcal{D}$. The following is then easily obtained from the properties of bracket of vector fields, Jacobi identity and maximal non--integrability of $\mathcal{D}$.

\begin{lem*}
The above tensorial brackets define a Lie bracket $[,]$ on the vector space $\fm:=\fg_{-l}\oplus \dots \oplus \fg_{-1}$ with the following properties:
\begin{enumerate}
\item $\fm$ is generated by $\fg_{-1},$
\item $[\fg_{a},\fg_{b}]=\fg_{a+b}$ ($\fg_l=0$ for $l<-l$)
\item the complex structure $I$ on $\fg_{-1}$ is satisfying $$[I(X),I(Y)]=[X,Y]$$ for all $X,Y\in \fg_{-1}.$
\end{enumerate}
Therefore the tuple $(\fm,I)$ is a (non--degenerate) pseudo--complex fundamental graded Lie algebra of $l$--th type.
\end{lem*}

The tuple $(\fm,I)$ defined in the above lemma is called Levi--Tanaka algebra of $(M,\mathcal{D},I)$ at $p$.

\subsection{Universal Levi--Tanaka algebras}\label{Sec22}

In this article, we will use a different characterization of the universal Levi--Tanaka algebras of $l$--th type than in \cite{T}. Suppose $(\fm,I)$ is a Levi--Tanaka algebra. Let us recall that if we complexify $\fm$, then we can decompose $\fg_{-1}\otimes \mathbb{C}$ to $\pm i$--eigenspaces $\fg_{-1}^{10},\fg_{-1}^{01}$ of $I$ and the condition $[I(X),I(Y)]=[X,Y]$ implies that these are abelian Lie subalgebras of $\fm$. To characterize the complexification of universal Levi--Tanaka algebras of $l$--th type, we firstly consider a free complex Lie algebra generated $\fg_{-1}\otimes \mathbb{C}$ and quotient it by an ideal determined by $l$ and the fact that $\fg_{-1}^{10},\fg_{-1}^{01}$ should be abelian Lie subalgebras.

Let us start by recalling the usual construction of the free complex Lie algebras generated by $\fg_{-1}\otimes \mathbb{C}$ using the Lyndon words. In the construction, we can fix arbitrary basis of $\fg_{-1}\otimes \mathbb{C}$. Therefore, we will assume that $z_1,\dots, z_n$ and $\bar{z}_1,\dots,\bar{z}_n$ are conjugated bases of $\fg_{-1}^{10}$ and $\fg_{-1}^{01}$, respectively and consider an ordered set $\mathcal{A}=\{z_1<\dots<z_n<\bar{z}_1<\dots<\bar{z}_n\}$. 

The set of Lyndon words over $\mathcal{A}$ is the set of all words from the alphabet $\mathcal{A}$ that are strictly smaller (w.r.t. to the lexicographic ordering) than any other cyclic permutation of them selves. For example, $z_1z_1\bar{z}_1z_1\bar{z}_1$ is a Lyndon word, but $z_1\bar{z}_1z_1\bar{z}_1$ is not a Lyndon word (cyclic permutation by two to either side gives the same word). It is well--known that the Lyndon words form a basis of the free Lie algebra $L(\mathcal{A})$ over the set $\mathcal{A}$. More precisely, the map $B$ that makes from a Lyndon word an element $L(\mathcal{A})$ is defined iteratively as follows:
\begin{itemize}
\item if $w$ is a letter in $\mathcal{A}$, then $B(w)=w$, and
\item if $w=uv$ for Lyndon words $u,v$ and this decomposition has the longest $v$ among all possible decompositions of $w$ to Lyndon words, then $B(w)=[B(u),B(v)]$.
\end{itemize}

If one replaces in the definition of the map $B$ the bracket in $L(\mathcal{A})$ by the bracket in the complexification of the non--degenerate pseudo--complex fundamental graded Lie algebra $(\fm,I)$, then one obtains a surjective Lie algebra homomorphism $L(\mathcal{A})\to \fm\otimes \mathbb{C}$. In \cite{T}, it is shown that there is universal Levi--Tanaka algebra of $l$--th type with complexification $L_{CR,l}(\mathcal{A})$ such that the map $L(\mathcal{A})\to \fm\otimes \mathbb{C}$ factors via the natural surjective Lie algebra homomorphisms $L(\mathcal{A})\to L_{CR,l}(\mathcal{A})$, i.e., it is given by a unique surjective Lie algebra homomorphism $L_{CR,l}(\mathcal{A})\to \fm\otimes \mathbb{C}$.

Clearly, the conditions that describe $L_{CR,l}(\mathcal{A})$ as a quotient of the complex free Lie algebra $L(\mathcal{A})$ are the length $l$ and the property that $z_i's$ (a thus $\bar{z}_i's$) commute with each other. This allows us to describe the basis of $L_{CR,l}(\mathcal{A})$ explicitly.

\begin{prop*}\label{Prop22}
The set of Lyndon words of length at most $l$ with the properties that:
\begin{enumerate}
\item the word begins with $z_i$ for some $i$ and any consecutive sequence of $z_i's$ is non--decreasing, and
\item the word ends with $\bar{z}_j$ for some $j$ and any consecutive sequence of $\bar{z}_i's$ is non--increasing,
\end{enumerate}
together with the set $\mathcal{A}$ form a basis of the Lie algebra $L_{CR,l}(\mathcal{A})$, i.e., of the complexification of the universal Levi--Tanaka algebra.
\end{prop*}
\begin{proof}
Firstly, we shot that if a Lyndon word does not satisfy the condition (1) or (2), then it is in the kernel of the projection $L(\mathcal{A})\to L_{CR,l}(\mathcal{A})$ up to a linear combination of words that satisfy the condition (1) or (2). 

If $w$ does not start with $z_i$ for some $i$ and has length at least two, than it contains only $\bar{z}_i's$ and clearly is in the kernel.

If $w$ contains sequence of  $z_i's$ is decreasing, then along the application of $b$ on $w$ one encounters a Lyndon word $u$ of length at least two ending with the (first) largest of this $z_i's$ that contains unique $z_j<z_i$. If $u=z_jz_i$, then $w$ is clearly in the kernel, otherwise, there is word $v$ such that $u=u_jvz_i$. Then using Jacobi identity, we get that $u=z_jz_iv$ modulo other basis vectors.

If $w$ contains increasing sequence of $\bar{z}_i's$, then along the application of $b$ on $w$ one encounters a Lyndon word $u$ of length at least $2$ containing only $\bar{z}_i's$ and clearly $w$ is in the kernel.
 
 If $w$ does not end with $\bar{z}_i$ and has length at least two, then along the application of $b$ on $w$ one encounters a Lyndon word $u$ of length at least two ending with $z_i$ that contains unique $z_j<z_i$. If $u=z_jz_i$, then $w$ is clearly in the kernel, otherwise, there is word $v$ such that $u=u_jvz_i$. Then using Jacobi identity, we get that $u=z_jz_iv$ modulo other basis vectors.

Thus it remains to show that the Lyndon words $w$ satisfying the properties (1) and (2) are linearly independent in $L_{CR,l}(\mathcal{A})$. It is clear that we can not use the Jacobi identity to express the words $w$ as linear combination of words satisfying the properties (1) and (2), because the Jacobi identity breaks the ordering. Thus it suffices to show that $w$'s are non--trivial elements of $L_{CR,l}(\mathcal{A})$. Suppose $w=uv$ for Lyndon words $u,v$ and this decomposition has the longest $v$ among all possible decompositions of $w$ to Lyndon words. Then from definition of $L_{CR,l}(\mathcal{A})$ follows that either:
\begin{itemize}
\item $u=z_i$, $v=\bar{z}_j$ and $[z_i,\bar{z}_j]$ is a non--trivial element of $L_{CR,l}(\mathcal{A})$, or
\item $u=z_i$ and $v$ satisfies the conditions (1) and (2) ($v$ not starting with $z_i$ for some $i$ is contradiction with $v$ being Lyndon word satisfying (2)) and $[z_i,B(v)]$ is a non--trivial element of $L_{CR,l}(\mathcal{A})$ if and only if $B(v)$ is a non--trivial element, or
\item  $v=\bar{z}_j$ and $u$ satisfies the conditions (1) and (2) ($u$ not ending with $\bar{z}_j$ for some $j$ is contradiction with $u$ being Lyndon word satisfying (1)) and $[B(u),\bar{z}_j]$ is a non--trivial element of $L_{CR,l}(\mathcal{A})$ if and only if $B(u)$ is a non--trivial element, or
\item $u$ and $v$ satisfy the conditions (1) and (2) and $[B(u),B(v)]$ is a non--trivial element of $L_{CR,l}(\mathcal{A})$ if and only if $B(u)$ and $B(v)$ are non--trivial elements.
\end{itemize}
\end{proof}

\subsection{Explicit embedding of the standard real submanifolds}\label{Sec23}

Consider a non--degenerate pseudo--complex fundamental graded Lie algebra $(\fm,I)$. Then it is simple observation that the nilpotent Lie group $\exp(\fm)$ together with the left invariant distribution and the complex structure given by $\fg_{-1}$ and $I$ is a CR manifold. In \cite[Section 1.1]{N}, there is a construction that embeds $\exp(\fm)$ into $\mathbb{C}^{N}$ as a CR submanifold, where $N=dim(\fm)-dim_\mathbb{C}(\fg_{-1})$. If we consider the following subalgebras of $\fm\otimes \mathbb{C}:$
$$\fg_{-1}^{10}:=\{X-iI(X): X\in \fg_{-1}\},$$
$$\fg_{-1}^{01}:=\{X+iI(X): X\in \fg_{-1}\},$$
$$\fn:=(\fg_{-l}\oplus \dots \oplus \fg_{-2})\otimes \mathbb{C}\oplus \fg_{-1}^{10},$$
then the embedding from \cite[Section 1.1]{N} is given as the inclusion $\exp(\fm)\subset \exp(\fn)$. As in \cite[Remark in Section 1.1]{N}, the exponential maps $\exp$ and $\exp^{-1}$ provide a coordinate expression $\phi: \fm\to \fn$ for this embedding. Let us write down the precise formula for $\phi$.

We observe that $\fn$ is an ideal of $\fm\otimes \mathbb{C}$ and that $$\exp(\fm\otimes \mathbb{C})\cong\exp(\fn)\cdot \exp(\fg_{-1}^{01}).$$
This means that for each $X\in \fm$, there is unique $Y\in \fg_{-1}^{01}$ such that $\phi(X)=\exp^{-1}(\exp(X)\exp(Y))$. Moreover, we can use the Baker--Campbell--Hausdorff formula to compute the expression $\exp^{-1}(\exp(X)\exp(Y))$ explicitly, because the Baker--Campbell--Hausdorff formula is finite and globally defined for the nilpotent Lie algebras.

\begin{lem*}
For a non--degenerate pseudo--complex fundamental graded Lie algebra $(\fm,I)$, if we decompose $X=X_{-l}+\dots+X_{-1}\in \fm=\fg_{-l}\oplus \dots \oplus \fg_{-1}$, then the embedding $\phi: \fm\to \fn$ is given by
\begin{align*}
\phi(X)&=\exp^{-1}(\exp(X)\exp(-\frac12(X_{-1}+iI(X_{-1})))))\\
&=\frac12(X_{-1}-iI(X_{-1}))\\
&+X_{-2}+i\frac14[I(X_{-1}),X_{-1}]\\
&+X_{-3}+f(X_{-2}+X_{-1},-\frac12(X_{-1}+iI(X_{-1})))_{-3}\\
&\vdots\\
&+X_{-l}+f(X_{-l+1}+\dots+X_{-1},-\frac12(X_{-1}+iI(X_{-1})))_{-l},
\end{align*}
where $f(Y,Z)_{-i}$ denotes the component of $$f(Y,Z):=\sum_{n=1}^{\infty}\frac{(-1)^{n}}{n+1}\sum_{r_i+s_i>0}\frac{ad(Y)^{r_1}ad(Z)^{s_1}\dots ad(Y)^{r_n}ad(Z)^{s_n}(Y)}{(1+\sum_{i=1}^nr_i)\prod_{i=1}^n r_i!s_i!}$$
 in $\fg_{-i}\otimes \mathbb{C}$.
\end{lem*}
\begin{proof}
The Baker--Campbell--Hausdorff formula has form
$\exp^{-1}(\exp(Y)\exp(Z))=Y+Z+f(Y,Z).$ If we decompose the formula according to decomposition $(\fg_{-l}\oplus \dots \oplus \fg_{-2})\otimes \mathbb{C}\oplus \fg_{-1}^{10}$ of $\fn$ and take in the account the properties of the bracket $[X,Y]=ad(X)(Y)$, then we obtain the claimed formula for the embedding $\phi$.
\end{proof}

We can observe that we can easily eliminate the variables from $\fm$ in the formula for $\phi: \fm\to \fn$ by solving linear equations. Moreover, from the properties of the non--degenerate pseudo--complex fundamental graded Lie algebras we can obtain the following corollary. 

\begin{cor*}\label{cor24}
If we assign the weight $j$ to the variables in $\fg_{-j}\otimes \mathbb{C}$, then the defining polynomials obtained from the line $X_{-i}+f(X_{-i+1}+\dots+X_{-1},-\frac12(X_{-1}+iI(X_{-1})))_{-i}$ are weighted homogeneous polynomials of weighted degree $i$. In particular, the standard real submanifold $\phi(\fm)\subset \fn$ corresponding to non--degenerate pseudo--complex fundamental graded Lie algebra $(\fm,I)$ of $l$--th type is a model surface.
\end{cor*}

\subsection{CR automorphisms of standard real submanifolds and the Tanaka prolongation}\label{Sec24}

The infinitesimal CR automorphisms of the standard real submanifold $\phi(\fm)\subset \mathbb{C}^N$ corresponding to the Levi--Tanaka algebra $(\fm,I)$ can be computed by the so called Tanaka prolongation from \cite{T}. Let us recall that the Tanaka prolongation of $(\fm,I)$ is the maximal graded Lie algebra $\fg=\oplus_{-l\leq i}\fg_i$ such that the negative part of grading is $\fm$, zero part $\fg_0$ of the grading commutes with $I$ and the maximal ideal of $\fg$ in the non--negative part of $\fg$ is trivial.

The result of \cite{T} is that
$$\fg_0=\{f\in \frak{der}_0(\fm)| f(I(Y))=I(f(Y)) \rm{\ for\ all\ } Y\in \fg_{-1}\}$$
$$\fg_i=\{f\in \oplus_{j<0}\fg_j^*\otimes \fg_{j+i}:\ f([X,Y])=[f(X),Y]+[X,f(Y)] \rm{\ for\ all\ } X,Y\in \fm\}$$
holds for the Tanaka prolongation $\fg$ of $(\fm,I)$, where $\frak{der}_0(\fm)$ is the space of the grading preserving derivations of $\fm$. Let us emphasize that $f\in \fg_i$ is uniquely determined by the component of $f$ in $\fg_{-1}^*\otimes \fg_{i-1}$, because $\fm$ is generated by $\fg_{-1}$. In \cite{T}, Tanaka proves the following result.

\begin{thm*}
If $(\fm,I)$ is a non--degenerate pseudo--complex fundamental graded Lie algebra, then $\fg_s=0$ for all $s$ large enough and the Tanaka prolongation $\fg$ of $(\fm,I)$ is a finite dimensional Lie algebra.
\end{thm*}

Let us now show, how the elements of the Tanaka prolongation $\fg$ of the Levi--Tanaka algebra $(\fm,I)$ are realized as the infinitesimal CR automorphisms of the standard real submanifold $\exp(\fm)\subset \mathbb{C}^N$:

If we identify $T_p\exp(\fm)\cong \fm$ via the left--invariant vector fields, then the value of the infinitesimal automorphism given by $X\in \fg$ at $p=\exp(Y),Y\in \fm$ is $(Ad(\exp(-Y))(X))_{\fm}$, where $(Z)_\fm$ is the component of $Z\in \fg$ in $\fm$. In the coordinates $\exp^{-1}$, we can apply the Baker--Campbell--Hausdorff formula on 
\begin{align*}
&\frac{d}{dt}|_{t=0}\exp(Y)\exp(t(Ad(\exp(-Y))(X))_{\fm})=\\
&\frac{d}{dt}|_{t=0}\exp(tAd(\exp(Y))(Ad(\exp(-Y))(X))_{\fm})\exp(Y)
\end{align*}
 to obtain the following coordinate formulas for infinitesimal automorphisms:

\begin{align*}
&\frac{d}{dt}|_{t=0}g(tX_{-l},0)\\
&\vdots\\
&\frac{d}{dt}|_{t=0}g(tX_{-1},Y_{-1}+\dots+Y_{-l+1}),\\
&\frac{d}{dt}|_{t=0}g(t(1-\Ad(\exp(Y))X_0,Y_{-1}+\dots+Y_{-l+1}),\\
&\vdots\\
&\frac{d}{dt}|_{t=0}g(t(1-\Ad(\exp(Y))(\sum_{l=0}^{j}\frac{1}{l!}ad(-Y)^lX_j),Y_{-1}+\dots+Y_{-l+1}),
\end{align*}
where $g(Z,Y):=Z+\sum_{n=1}^{\infty}\frac{(-1)^{n}}{n+1}\sum_{s_i>0}\frac{ad(Y)^{s_1+\dots+s_{n}}(Z)}{\prod_{i=1}^{n} s_i!}$ is the part of $Z+Y+f(Z,Y)$ linear in $Z$. Note that $ad(Y)^{s_1+\dots+s_{n}}(Z)=0$ for $n$ and $s_i$ large enough.

We can use analogous construction to construct all holomorphic vector fields on $\fn$ preserving the standard real submanifold $\phi(\fm)\subset \mathbb{C}^N=\fn$.

If we identify $T_p\exp(\fn)\cong \fn$ via the left--invariant vector fields, then value of the infinitesimal automorphism given by $X\in \fg$ at $p=\exp(Y),Y\in \fn$ is $(Ad(\exp(-Y))(X))_{\fn}$, where $(Z)_\fn$ is the component of $Z\in \fg$ in $\fm\otimes \mathbb{C}$ projected along $\fg_{-1}^{01}$ into $\fn$. In the coordinates $\exp^{-1}$, we can apply the Baker--Campbell--Hausdorff formula on 
\begin{align*}
&\frac{d}{dt}|_{t=0}\exp(Y)\exp(t(Ad(\exp(-Y))(X))_{\fn})=\\
&\frac{d}{dt}|_{t=0}\exp(tAd(\exp(Y))(Ad(\exp(-Y))(X))_{\fn})\exp(Y)
\end{align*}
 to obtain the following coordinate formulas for infinitesimal automorphisms:

\begin{align*}
&\frac{d}{dt}|_{t=0}g(tX_{-l},0)\\
&\vdots\\
&\frac{d}{dt}|_{t=0}g(tX_{-2},Y_{-1}+\dots+Y_{-l+2}),\\
&\frac{d}{dt}|_{t=0}g(tAd(\exp(Y))(Ad(\exp(-Y))(X_{-1}))_{\fn},Y_{-1}+\dots+Y_{-l+1}),\\
&\vdots\\
&\frac{d}{dt}|_{t=0}g(tAd(\exp(Y))(Ad(\exp(-Y))(X_j))_{\fn},Y_{-1}+\dots+Y_{-l+1}).
\end{align*}

Since expressions $\frac{d}{dt}|_{t=0}(tX_{j})$  for $j<0$ are linear combinations of the coordinate vector fields on $\fn$, we can naturally attach the weight $j$ to them and obtain the following result.

\begin{cor*}\label{cor2.5}
For $X\in \fg_j$, the corresponding infinitesimal CR automorphism is a weighted homogeneous holomorphic vector field on $\mathbb{C}^N=\fn$ of weighted degree $j$.
\end{cor*}

\subsection{The proof of Theorem \ref{Thm1}}\label{Sec3}

The results in the the previous sections prove that claim (2) implies claim (1) of the Theorem \ref{Thm1}, because the standard real submanifold $\phi(\fm)\subset \mathbb{C}^N$ is a model surface.

To prove that (1) implies (2), is suffices to show that a model surface $M_0$ is isomorphic to the standard real submanifold $\phi(\fm)\subset \mathbb{C}^N$ corresponding to the Levi--Tanaka algebra $(\fm,I)$ of $M_0$ at $0$.

Each model surface $M_0$ has a one parameter group of CR automorphisms acting as $z_i\mapsto z_i\exp(t),\ w_i\mapsto w_i\exp(t\cdot [w_i])$ for $t\in \mathbb{R}$. This implies that the infinitesimal automorphisms decompose according to the weighted degree for $[\frac{\partial}{\partial z_i}]=-1,\ [\frac{\partial}{\partial w_i}]=-[w_i]$ and are polynomial. In particular, it is easy to check using homogeneity w.r.t. CR automorphism group that the complex tangent space $T^{\mathbb{C}}M_0$ is spanned by infinitesimal CR automorphisms of degree $-1$. Let us denote $\mathfrak{inf}(M_0)_i$ for the space of infinitesimal CR automorphisms of $M_0$ of weighted degree $i$, which clearly is a grading of the Lie algebra of infinitesimal CR automorphisms $\mathfrak{inf}(M_0)$ of $M_0$. The definition of the Levi--Tanaka algebra implies that space of infinitesimal CR automorphisms of negative weighted degree is isomorphic as a graded Lie algebra to the Levi--Tanaka algebra of $M_0$ at $0$. In particular, the above one parameter group corresponds to the exponential group of the grading element of $\fm$.

Consequently, the Lie algebra $\mathfrak{inf}(M_0)$ itself is a graded Lie subalgebra of the Tanaka prolongation $\fg$ of the Levi--Tanaka algebra $(\fm,I)$, because the infinitesimal CR automorphisms of $M_0$ of non--negative weighted degree can be identified with an element of $\fg$ just by their action on $\fg_-$. Indeed, the main property of the Tanaka equivalence method from \cite{T} is that this map $\mathfrak{inf}(M_0)\to \fg$ is always injective linear map that induces a Lie algebra homomorphisms for the graded Lie algebras.

Since both the model surface $M_0$ and the standard real submanifold $\phi(\fm)\subset \mathbb{C}^N$ corresponding to the Levi--Tanaka algebra $(\fm,I)$ are homogeneous and real analytic, the inclusion  $\mathfrak{inf}(M_0)\subset \fg$ clearly defines the CR isomorphism of these two and thus the claim (1) implies claim (2) of the Theorem \ref{Thm1}.

The remaining claims are direct consequence of the results in the previous sections. Indeed, there is always surjective homomorphisms from the universal Levi--Tanaka algebra to any Levi--Tanaka algebra of the same type and from Corollary \ref{cor24} follows that  the model surface is totally nondegenerate if and only if the claimed condition is satisfied.

\section{Proof of the Main theorem}\label{Sec4}

In this section, we prove the Theorem \ref{Thm2}, which together with the results of \cite{G} completely resolves the Beloshapka's conjecture and proves the Main theorem and Theorem \ref{Thm3}.

A direct consequence of the Theorem \ref{Thm1} is that the two conditions (1) and (2) of the Theorem \ref{Thm2} are equivalent. So we suppose that $\fg$ is the Tanaka prolongation of a non--degenerate pseudo--complex fundamental graded Lie algebra $(\fm,I)$ of $l$--th type for $l>3$ satisfying for some $K$ such that  $l+2\leq 2K$ that the dimensions of spaces $\fg_{-2}, \dots,\fg_{-K}$ in the Levi--Tanaka algebra of $l$--th type of $M_0$ at $0$ are the same as dimensions of the corresponding grading components of the universal Levi--Tanaka algebra.

We know from \cite[Theorem 3.27]{M} that there is a Levi decomposition $\fg=\frak{s}\oplus \frak{r}$ compatible with the structure of the non--degenerate pseudo--complex fundamental graded Lie algebra $(\fm,I)$ and the grading of the Tanaka prolongation $\fg$. This means that $\frak{s}$ is a graded semisimple Lie subalgebra of $\fg$ such that $(\frak{s}_-,I)$ is non--degenerate pseudo--complex fundamental graded Lie algebra, where the fundamentality follows from \cite[Proposition 3.2]{M}. Let us remark that $\frak{s}_{-1}$ is isomorphic to $\fg_{-1}/\fr_{-1}$ with the induced complex structure $I$ on the quotient. On the other hand, since $\frak{s}_-$ is subalgebra of $\fg_-$ generated by $\frak{s}_{-1}$, the dimensions of $\frak{s}_{-2},\dots,\frak{s}_{-K}$ correspond by our assumptions with the dimension of the corresponding grading component of universal Levi--Tanaka algebra generated by $\frak{s}_{-1}$. Since $K\geq 3$, this means that the dimension of $\frak{s}_{-3}$ has to be $dim(\frak{s}_{-1})^3+dim(\frak{s}_{-1})^2$, which exceeds the possible dimensions of $\frak{s}_{-3}$ known from the classification Levi--Tanaka algebras with semisimple Tanaka prolongation. We recall this classification in Section \ref{Sec5}. Therefore  $dim(\frak{s}_{-1})=0$ and $\frak{s}\subset \fg_0$. Thus $\fm=\frak{r}_-$ and $\fg_+=\fr_+$ and the grading element of $\fg$ is also contained in $\fr$. Therefore, $\frak{g}_{-1}\oplus \frak{g}_1\subset [\frak{r},\frak{r}]$ and the Lie subalgebra of $\fg$ generated by $\frak{g}_{-1}\oplus \frak{g}_1$ is nilpotent.

Therefore if $\fg_0$ is reductive (which is the case when CR--dimension $n=1$ or when $\fg_-$ is the universal Levi--Tanaka algebra), then $[\frak{g}_{-1},\frak{g}_1]=0$ and thus $\fg_1=\fg_+=0$. 

For general $\fg_0$, we consider the following real nilpotent Lie subalgebra $N(a,b,Y)$ of the complexification $\fr\otimes \mathbb{C}$ generated by two conjugated elements $a,b\in \fg^{10},\fg^{01}$ and an element $Y\in \frak{g}_{1}\subset \fr\otimes \mathbb{C}:$ 

We pick arbitrary $Y\in \fg_1$ that is highest in the components of the lower central series for the nilpotent Lie algebra $[\frak{r},\frak{r}]$. 

We pick $a=X-iI(X)$ for the following $X\in \fg_{-1}$. Firstly, we pick arbitrary $X_1\in \frak{g}_{-1}$ such that $ad(X_1-iI(X_1))^{l}Y\neq 0$ for some $l>0$. If $ad(X_1-iI(X_1))^2Y\neq 0$, then we choose $X=X_1$. If $ad(X_1-iI(X_1))^2Y=0$, then there is $X_2\in \frak{g}_{-1}$ such that $[X_2-iI(X_2),[X_1-iI(X_1),Y]]\neq 0$. If $[X_2-iI(X_2),Y]=0$, then we choose $X=X_1+X_2$. If $ad(X_2-iI(X_2))^2Y\neq 0$, then we choose $X=X_2$.  If $ad(X_2-iI(X_2))^2Y= 0$, then $[X_1-iI(X_1)+X_2-iI(X_2),[X_1-iI(X_1)+X_2-iI(X_2),Y]]=2[X_2-iI(X_2),[X_1-iI(X_1),Y]]+[[X_1-iI(X_1),X_2-iI(X_2)],Y]=2[X_2-iI(X_2),[X_1-iI(X_1),Y]]\neq 0$ and we we choose $X=X_1+X_2$.

We denote $Z:=[Y,a]$ and consider the elements 
$$a_l:=ad(Z)^la, b_l:=ad(Z)^lb$$
of $N(a,b,Y)\cap \fg_{-1}\otimes \mathbb{C}$, which (if non--zero) are linearly independent by nilpotency. Therefore we can choose $a_l$'s a part of the basis of $\fg_{-1}^{10}$ in the complexification of universal Levi--Tanaka algebra from Proposition \ref{Prop22}. Let us emphasize that $a_1\neq 0$ by the choice of $a$.

\begin{lem*}\label{tec1}
The words $a_{j_1}\dots a_{j_{p-1}}b_{j_p}$ of the length $1<p\leq K$ satisfying $a_l\leq a_{l+1}$ naturally represent a sum of Lyndon words satisfying the conditions (1) and (2) of Proposition \ref{Prop22}. Moreover,  all these words (if non--zero) are linearly independent as elements of $N(a,b,Y)$.
\end{lem*}
\begin{proof}
By definition, the elements $b_l$ can be expressed as linear combination of $\bar{a}_l$'s and the renaming barred elements of the basis of $\fg_{-1}^{01}$. Therefore the first claim is clear. The second claim follows from our assumption that the dimensions of spaces $\fg_{-2}, \dots,\fg_{-K}$ in the Levi--Tanaka algebra of $l$--th type of $M_0$ at $0$ are the same as dimensions of the corresponding grading components of the universal Levi--Tanaka algebra and from the linear independency of $b_l$'s.
\end{proof}

We will need the following relations:
$$[Y,Z]=0,$$
which holds by the choice of $Y$ and
$$[Y,a_{l+1}]=ad(Z)^l[Y,a]=ad(Z)^lZ=0,$$
$$[a_{l_1},[a_{l_2},[Y,b_j]]]=0,$$
which hold, because $[Y,b_j]\in \fg_0\otimes \mathbb{C}$ acts in complex linear way. These relations prove the following technical lemma.

\begin{lem*}\label{tec}
$$[Y,a_{i_1}\dots a_{i_p}b_j]=\sum_{j_1\leq \dots\leq j_p} c_{j_1,\dots,j_{p}} a_{j_1}\dots a_{j_{p-1}}b_{j_p}$$
for $p>1$ and some non--negative integers $c_{j_1,\dots,j_{p}}$. If $a_{i_l}a_{i_{l+1}}a_{i_{l+2}}=aaa_{i_{l+2}}$ for $i_{l+2}>0$, then $c_{a_1,\dots,a_{l-1},1,a_{l+2},\dots,a_p,j}>0$ and $c_{a_1,\dots,a_{l},a_{l+2},\dots,a_p,j+1}>0$.
\end{lem*}
\begin{proof}
If $a_{i_1}\neq a$, then $[Y,a_{i_1}\dots a_{i_p}b_j]=0$ follows by using the Jacoby identity and the above relations.

If $a_{i_1}=a$, then by the Jacoby identity 
$$[Y,a_{i_1}\dots a_{i_p}b_j]=[Z,a_{i_2}\dots a_{i_p}b_j]+[a,[Y,a_{i_2}\dots a_{i_p}b_j]].$$ By definition, $[Z,a_{i_2}\dots a_{i_p}b_j]$ can be expressed in the claimed way, because $a$'s commute. The second summand can be expressed in the claimed way by the iterative use of the Jacoby identity, because either  $ad(a)^{l-1}[Z,a_{i_l}\dots a_{i_p}b_j]$ can be expressed in the claimed way or $[Y,a_{i_l}\dots a_{i_p}b_j]=0$ for some $l$ or $ad(a)^p[Y,b_j]=0$ and there are no other types of summands appearing after sufficient number of applications of the Jacobi identity. Moreover, if we arrive in the iteration to the situation  $[Y,aaa_{i_{l+2}}\dots a_{i_p}b_j]$, then $c_{a_1,\dots,a_{l-1},1,a_{l+2},\dots,a_p,j}>0$, because $[Y,aaa_{i_{l+2}}\dots a_{i_p}b_j]=a_1a_{i_{l+2}}\dots a_{i_p}b_j+[a,[Z,a_{i_{l+2}}],\dots a_{i_p}b_j]]+\dots+aa_{i_{l+2}}\dots a_{i_p}b_{j+1}.$
\end{proof}

The element $a^{l}b\in \fg_{-l-1}$ is by assumption $0$. If $l= 2s-2$, then $0=ad(Y)^{s-1}a^{l}b$ can be expressed according to the Lemma \ref{tec} as a non--negative combinations of the elements $a_{j_1}\dots a_{j_{s-1}}b_{j_s}$. Since $s\leq K$, these elements (if non--zero) are linear independent due to  Lemma \ref{tec1} and this combination has to have all coefficients $0$.  However, if $\fg_1\neq 0$, then the coefficient by $a_1^{k-1}b$ (which is non--zero by the choice of $X,Y$) is positive due to Lemma \ref{tec}, contradiction.

If $l=2s-3$, then $0=ad(Y)^{s-2}a^{l}b$ can be expressed according to the Lemma \ref{tec} as a non--negative combinations of the elements $a_{j_1}\dots a_{j_{s-1}}b_{j_s}$. Since $s\leq K$, these elements are linear independent (if non--zero) due to  Lemma \ref{tec1} and this combination has to have all coefficients $0$. However, if $\fg_1\neq 0$, then the coefficient by $aa_1^{k-2}b$ (which is non--zero by the choice of $X,Y$) is positive due to Lemma \ref{tec}, contradiction.

Therefore $\fg_1=\fg_+=0$.

\section{Standard real submanifolds with Tanaka prolongation that is simple complex Lie algebra}\label{Sec5}

In this Section, we recall the classification of all non--degenerate pseudo--complex fundamental graded Lie algebra $(\fm,I)$ that have a simple complex Lie algebra $\fg$ as the Tanaka prolongation. The full classification all non--degenerate pseudo--complex fundamental graded Lie algebra $(\fm,I)$ with semisimple Tanaka prolongation can be found in \cite{MS}, but two main results of the classification are that semisimple Tanaka prolongation is a sum of simple Tanaka prolongation and that a complexification of real Tanaka prolongation is a Tanaka prolongation of the complexification of  $(\fm,I)$. Since we are mainly interested in the dimension of $\fg_{-3}$, it suffices to recall just the cases when the Tanaka prolongation is a simple complex Lie algebra and for this, it is convenient to use the results from \cite{A}, because we can easily summarize this classification (up to symmetries of Dynkin diagrams) of all such non--degenerate pseudo--complex fundamental graded Lie algebra $(\fm,I)$ in Table \ref{T1}. 

\begin{table}\caption{Pseudo--complex fundamental graded Lie algebras with simple complex Tanaka prolongation}\label{T1}
\noindent
\begin{tabular}{|c|c|c|c|}
\hline
$\mathfrak{g}$ & $\Xi$ &$l$& restrictions\\
\hline
$A_n$&$\{a_1,\dots, a_s\}$&  $=s$ & \\
\hline
$B_n$& $\{a_1,\dots, a_s\}$& $\geq2s-1$ & $a_{s-1}=a_s-1$\\
\hline
$C_n$& $\{a_1,\dots, a_s\}$&$=2s-1$ &$a_s=n$\\
\hline
$D_n$& $\{a_1,\dots, a_s\}$&  $\geq 2s-1$ & $a_{s-1}=a_s-1,a_s<n-1$\\
& $\{a_1,\dots, a_s\}$& $\geq 2s-2$& $a_s=n-1$\\
&$\{a_1,\dots, a_s\}$& $\geq 2s-3$& $a_{s-2}=n-2,a_{s-1}=n-1,a_s=n$\\
\hline
$G_2$ & $\{1,2\}$&$=5$&\\
\hline
$F_4$ & $\{a_1,\dots, a_s\}$&$\geq 6$&$a_b=2,$ except $\{2,4\}$\\
\hline
$E_6$ & $\{a_1,\dots, a_s\}$&$\geq 6$&$s>3,a_b=3$ or $\{1,2,4,5\}$\\
 & $\{a_1,a_2,a_3\}$&$\geq 4$&except $\{1,3,5\},\{1,4,6\},\{1,5,6\},\{2,4,6\}$ \\
 & $\{a_1,a_2\}$&$\geq 2$&except $\{1,3\},\{1,4\}$ \\
\hline
$E_7$ & $\{a_1,\dots, a_s\}$&$\geq 11$ &$s>4,a_b=4$ or $\{1,2,3,5,6\}$\\
& $\{a_1,\dots, a_s\}$&$\geq 8$&$s=4$, except $\{1,2,4,6\}$, $\{a_1,a_2,a_3,7\}$\\
&&& for $a_1\in \{1,2,3\},a_3\in \{5,6\}$   \\
 & $\{a_1,a_2,a_3\}$&$\geq 5$&except $\{1,2,5\},\{1,3,6\},\{1,2,4\}, \{1,4,6\},$ \\
&&&$ \{2,4,6\},\{a_1,a_2,7\}$ for $a_1\in \{1,2,3\},a_2\in \{5,6\}$\\
 & $\{a_1,a_2\}$&$\geq 3$&except $\{1,3\},\{1,4\},\{1,5\},\{1,7\},$\\
&&&$\{2,4\},\{2,5\},\{2,6\},\{3,6\},\{4,6\}$\\
\hline
$E_8$ & $\{a_1,\dots, a_s\}$&$\geq 20$ &$s>5,a_b=5$ or $\{1,2,3,4,6,7\}$\\
 & $\{a_1,\dots, a_s\}$&$\geq 15$&$s=5$, except $\{1,2,3,5,7\}, \{a_1,a_2,a_3,a_4,8\}$ \\
&&&for $a_1\in \{1,2,3,4\},a_4\in \{6,7\}$    \\
 & $\{a_1,\dots, a_s\}$&$\geq 11$&$s=4$, except $\{1,2,5,7\},\{1,3,5,7\},\{2,3,5,7\},$\\
&&&$\{1,2,3,5\},\{1,2,3,6\},\{1,2,4,7\}, \{a_1,a_2,a_3,8\}$\\
&&&for $a_1\in \{1,2,3,4\},a_3\in \{6,7\}$   \\
 & $\{a_1,a_2,a_3\}$&$\geq 7$&except $\{1,2,5\}, \{1,3,5\},\{2,3,5\},\{1,3,7\},$\\
&&&$\{1,5,7\},\{2,5,7\},\{3,5,7\},\{1,2,6\},\{1,3,6\},$\\
&&&$\{2,3,6\},\{1,2,4\},\{1,4,7\},\{2,4,7\},\{1,2,8\}$ \\
&&&$\{a_1,a_2,8\}$ for $a_1\in \{1,2,3,4\},a_2\in \{6,7\}$  \\
 & $\{a_1,a_2\}$&$\geq 5$&except $\{1,a_2\},a_2\neq 2, \{2,a_2\},a_2\neq 3$, \\
&&& $\{3,a_2\},a_2\notin \{4,8\}, \{a_1,7\}, a_1\in \{4,5\}$,   \\
\hline
\end{tabular}
\end{table}

In the table, the column $\fg$ specifies the complex Lie algebra $\fg$ with arbitrary but fixed set of positive simple roots that we number according to \cite{P}. The column $\Xi$ specifies the $|l|$--grading of $\fg$ such that the Lie algebra $\fm$ is the negative part of the grading for arbitrary $s>1$ (if there are no other restrictions). Precisely, the grade of each root space of $\fg$ is given by a height functional on the root system that is $1$ on positive simple roots in $\Xi$ and is $0$ on the remaining simple roots. Consequently the Cartan subalgebra of $\fg$ is contained in $\fg_0$ and $\fm$ is the sum of the root spaces of roots of negative height. In particular, each root space in $\fg_{-1}$  corresponds to some simple root in $\Xi$. The value of $l$ is the height of the highest root and we present the lower bound on $l$ depending on $|\Xi|$.

The gradings defined by sets $\Xi$ in the Table \ref{T1} have the property that the set $\Xi$ can be decomposed to the sets $\Xi^+$ and $\Xi^-$ such that the elements from $\Xi^+$ and $\Xi^-$ alternate, when one moves along the edges of the Dynkin diagram (ignoring the roots outside of $\Xi$). The complex structure $I$ on $\fg_{-1}$ such that $(\fm,I)$ is a non--degenerate pseudo--complex fundamental graded Lie algebra is then defined as acting by $\pm i$ eigenvalue on the roots space corresponding simple roots in $\Xi^+$ or $\Xi^-$, respectively.

The results in the Section \ref{Sec2} allow us to explicitly compute the defining equations (\ref{bg}) and the infinitesimal CR automorphisms of the standard real submanifolds corresponding to the entries of the Table \ref{T1}. Let us recall that all gradings of simple Lie algebras are symmetric in the sense that $dim(\fg_{-i})=dim(\fg_i)$, cf. \cite{P}. Therefore the Theorem \ref{Thm3} follows immediately from this classification and the Corollary \ref{cor2.5}. The further details on on parabolic geometries can be found in \cite{P}.

\section{An example of model surface with infinitesimal CR automorphisms of positive weighted degree}\label{Sec51}

Consider a model surface in $\mathbb{C}^6$ given by the following defining equations:
\begin{align*}
&\mbox{Im}(w_{1})=z_{1}\bar z_{2}+\bar z_{1}z_{2},\\
&\mbox{Im}(w_{2})=-iz_{1}\bar z_{2}+i\bar z_{1}z_{2},\\
&\mbox{Im}(w_{3})=-z_1^2\bar z_{2}-\bar z_{1}^2z_2+\frac12(\mbox{Re}(w_{2})+i\mbox{Re}(w_{1}))\bar z_{1}+\frac12(\mbox{Re}(w_{2})-i\mbox{Re}(w_{1}))z_1,\\
&\mbox{Im}(w_{4})=i(z_1^2\bar z_{2}-\bar z_{1}^2z_2+\frac12(\mbox{Re}(w_{2})+i\mbox{Re}(w_{1}))\bar z_{1}-\frac12(\mbox{Re}(w_{2})-i\mbox{Re}(w_{1}))z_1),
\end{align*}
where $[z_1]=[z_2]=1, [w_1]=[w_2]=2, [w_3]=[w_4]=3$, i.e., the Bloom--Graham type is $(2,2,3,3)$, which is not totally non--degenerate. It is a standard real submanifold corresponding to the entry $\fg=\frak{sp}(4,\mathbb{C}),\Xi=\{1,2\}$ of the Table \ref{T1}. This means that Lie algebra of infinitesimal CR automorphisms of this model surface is isomorphic to $\frak{sp}(4,\mathbb{C})$ and is $|3|$--graded. Below, we show how to compute that the algebra infinitesimal CR automorphism of this model surface consists of the following weighted homogeneous (for $[\frac{\partial}{\partial z_1}]=[\frac{\partial}{\partial z_2}]=-1,$ $[\frac{\partial}{\partial w_1}]=[\frac{\partial}{\partial w_2}]=-2$, $[\frac{\partial}{\partial w_3}]=[\frac{\partial}{\partial w_4}]=-3$) holomorphic vector fields on $\mathbb{C}^6$:

\begin{align*}
&\fg_{-3}=\left \{\frac{\partial}{\partial w_3},\frac{\partial}{\partial w_4} \right \},\\
&\fg_{-2}=\left \{\frac{\partial}{\partial w_1}+z_1\frac{\partial}{\partial w_3}-iz_1\frac{\partial}{\partial w_4},\frac{\partial}{\partial w_2}+iz_1\frac{\partial}{\partial w_3}+z_1\frac{\partial}{\partial w_4}\right \},\\
&\fg_{-1}=\left \{ \frac12\frac{\partial}{\partial z_2}+iz_1\frac{\partial}{\partial w_1}+z_1\frac{\partial}{\partial w_2},\frac{i}2\frac{\partial}{\partial z_2}+z_1\frac{\partial}{\partial w_1}-iz_1\frac{\partial}{\partial w_2},\right.  \\
&\frac12\frac{\partial}{\partial z_1}+iz_2\frac{\partial}{\partial w_1}-z_2\frac{\partial}{\partial w_2}-\frac{3w_1-iw_2}{2}\frac{\partial}{\partial w_3}-\frac{3w_2+iw_1}{2}\frac{\partial}{\partial w_4},\\
&\left. \frac{i}2\frac{\partial}{\partial z_1}+z_2\frac{\partial}{\partial w_1}+iz_2\frac{\partial}{\partial w_2}+\frac{3w_2+iw_1}{2}\frac{\partial}{\partial w_3}-\frac{3w_1-iw_2}{2}\frac{\partial}{\partial w_4}\right \},\\
&\fg_{0}=\left \{
-z_1\frac{\partial}{\partial z_1}-w_1\frac{\partial}{\partial w_1}-w_2\frac{\partial}{\partial w_2}-2w_3\frac{\partial}{\partial w_3}-2w_4\frac{\partial}{\partial w_4} ,\right.\\
&z_1\frac{\partial}{\partial z_1}-2z_2\frac{\partial}{\partial z_2}-w_1\frac{\partial}{\partial w_1}-w_2\frac{\partial}{\partial w_2},\\
&-iz_1\frac{\partial}{\partial z_1}+w_2\frac{\partial}{\partial w_1}-w_1\frac{\partial}{\partial w_2}+2w_4\frac{\partial}{\partial w_3}-2w_3\frac{\partial}{\partial w_4},\\
&\left. iz_1\frac{\partial}{\partial z_1}+2iz_2\frac{\partial}{\partial z_2}+w_2\frac{\partial}{\partial w_1}-w_1\frac{\partial}{\partial w_2} \right\},\\
\end{align*}

\begin{align*}
&\fg_{1}=\left\{
-2z_1^2\frac{\partial}{\partial z_1}-(w_2+iw_1)\frac{\partial}{\partial z_2}-(w_3+iz_1w_2+z_1w_1)\frac{\partial}{\partial w_1}\right. \\
&-(w_4+z_1w_2-iz_1w_1)\frac{\partial}{\partial w_2}-(iz_1w_4+z_1w_3)\frac{\partial}{\partial w_3}+(iz_1w_3-z_1w_4)\frac{\partial}{\partial w_4},\\
&-2iz_1^2\frac{\partial}{\partial z_1}-(w_1-iw_2)\frac{\partial}{\partial z_2}+(w_4-iz_1w_1+z_1w_2)\frac{\partial}{\partial w_1}\\
&-(w_3+iz_1w_2+z_1w_1)\frac{\partial}{\partial w_2}+(z_1w_4-iz_1w_3)\frac{\partial}{\partial w_3}-(z_1w_3+iz_1w_4)\frac{\partial}{\partial w_4},\\
&\frac{w_2-iw_1}{2}\frac{\partial}{\partial z_1}-2z_2^2 \frac{\partial}{\partial z_2}+ (-w_1+iw_2)z_2 \frac{\partial}{\partial w_1}-z_2(iw_1+w_2) \frac{\partial}{\partial w_2}\\
&-\frac{i(-w_1+iw_2)^2}{2} \frac{\partial}{\partial w_3}+ \frac{(-w_1+iw_2)^2}{2} \frac{\partial}{\partial w_4},\\
&-\frac{iw_2+w_1}{2}\frac{\partial}{\partial z_1}-2iz_2^2 \frac{\partial}{\partial z_2}+ i(-w_1+iw_2)z_2 \frac{\partial}{\partial w_1}-z_2(-w_1+iw_2) \frac{\partial}{\partial w_2}\\
&\left.+ \frac{(-w_1+iw_2)^2}{2} \frac{\partial}{\partial w_3}+ \frac{i(-w_1+iw_2)^2}{2} \frac{\partial}{\partial w_4}\right \},\\
&\fg_{2}=\left\{
\frac{w_3+iw_4-z_1(w_1+iw_2) }{2}\frac{\partial}{\partial z_1}-2z_2(w_1-iw_2)\frac{\partial}{\partial z_2}\right. \\
&+(-w_1^2+w_2^2+z_2(iw_3+w_4))\frac{\partial}{\partial w_1}+(-2w_2w_1-z_2w_3+iz_2w_4)\frac{\partial}{\partial w_2}\\
&+\frac{(iw_2-3w_1)w_3+(3w_2+iw_1)w_4}{2}\frac{\partial}{\partial w_3}-\frac{(3w_2+iw_1)w_3+(3w_1-iw_2)w_4}{2}\frac{\partial}{\partial w_4},\\
&-\frac{w_4-iw_3-z_1(w_2-iw_1) }{2}\frac{\partial}{\partial z_1}+2z_2(w_2+iw_1)\frac{\partial}{\partial z_2}\\
&+(2w_2w_1+z_2w_3-iz_2w_4)\frac{\partial}{\partial w_1}+(-w_1^2+w_2^2+iz_2w_3+z_2w_4)\frac{\partial}{\partial w_2} \\
&\left.+\frac{(3w_2+iw_1)w_3+(3w_1-iw_2)w_4}{2}\frac{\partial}{\partial w_3}+\frac{(iw_2-3w_1)w_3+(3w_2+iw_1)w_4}{2}\frac{\partial}{\partial w_4}\right \},\\
&\fg_{3}=\left\{
-z_1(iw_4+w_3)\frac{\partial}{\partial z_1}-\frac{i(w_1-iw_2)^2}{2}\frac{\partial}{\partial z_2}-(w_1w_3-w_2w_4)\frac{\partial}{\partial w_1}\right. \\
&-(w_1w_4+w_2w_3)\frac{\partial}{\partial w_2}-(w_3^2-w_4^2)\frac{\partial}{\partial w_3}-2w_3w_4\frac{\partial}{\partial w_4},\\
&z_1(w_4-iw_3)\frac{\partial}{\partial z_1}-\frac{(w_1-iw_2)^2}{2}\frac{\partial}{\partial z_2}+(w_1w_4+w_2w_3)\frac{\partial}{\partial w_1}\\
&\left.-(w_1w_3-w_2w_4)\frac{\partial}{\partial w_2}+2w_3w_4\frac{\partial}{\partial w_3}-(w_3^2-w_4^2)\frac{\partial}{\partial w_4}\right \},
\end{align*}

To construct this example starting from the entry $\fg=\frak{sp}(4,\mathbb{C}),\Xi=\{1,2\}$ of the Table \ref{T1}, we consider $\frak{sp}(4,\mathbb{C})$ as the Lie algebra of derivations preserving the anti--symmetric bilinear form $((a_1,\dots,a_4),(b_1,\dots,b_4))\mapsto a_1b_4-a_4b_1+a_2b_3-a_3b_2$ on $\mathbb{C}^4$, which we identify with $\mathbb{R}^{8}$. This means that for $\Xi=\{1,2\}$ the Lie algebra $\fm$ consist of lower triangular matrices in $\frak{sp}(4,\mathbb{C})$ and the corresponding $|3|$--grading of $\frak{sp}(4,\mathbb{C})$ induces the following grading of $\fm$:
\[\begin{pmatrix}0 & 0 & 0 & 0 & 0 & 0 & 0 & 0 \\
0 & 0 & 0 & 0 & 0 & 0 & 0 & 0 \\
{{x}_{-1,1}} & -{{x}_{-1,2}} & 0 & 0 & 0 & 0 & 0 & 0 \\
{{x}_{-1,2}} & {{x}_{-1,1}} & 0 & 0 & 0 & 0 & 0 & 0 \\
{{x}_{-2,5}} & -{{x}_{-2,6}} & {{x}_{-1,3}} & -{{x}_{-1,4}} & 0 & 0 & 0 & 0 \\
{{x}_{-2,6}} & {{x}_{-2,5}} & {{x}_{-1,4}} & {{x}_{-1,3}} & 0 & 0 & 0 & 0 \\
{{x}_{-3,7}} & -{{x}_{-3,8}} & {{x}_{-2,5}} & -{{x}_{-2,6}} & -{{x}_{-1,1}} & {{x}_{-1,2}} & 0 & 0 \\
{{x}_{-3,8}} & {{x}_{-3,7}} & {{x}_{-2,6}} & {{x}_{-2,5}} & -{{x}_{-1,2}} & -{{x}_{-1,1}} & 0 & 0 \\
\end{pmatrix},\]
where all $x_{a,b}\in \fg_a$ are real. The complex structure $I$ on $\fg_{-1}$ corresponds to $i$ action on ${x}_{-1,1}+i{x}_{-1,2}$ and ${x}_{-1,4}+i{x}_{-1,3}$. Therefore the element $-\frac12(X_1+iI(X_1))\in \fg^{01}_{-1}$ that is necessary for the construction of the Naruki's embedding has $\frac{i{x}_{-1,2}-{x}_{-1,1}}{2}$ instead of ${x}_{-1,1}$, $-\frac{{x}_{-1,2}+i{x}_{-1,1}}{2}$ instead of ${x}_{-1,2}$, $-\frac{i{x}_{-1,4}+{x}_{-1,3}}{2}$ instead of ${x}_{-1,3}$, $ \frac{i{x}_{-1,3}-{x}_{-1,4}}{2}$ instead of ${x}_{-1,4}$ and $0$ elsewhere. Therefore the Lie subalgebra $\fn$ of $\fm\otimes\mathbb{C}$ is the following graded Lie algebra:
\[\begin{pmatrix}0 & 0 & 0 & 0 & 0 & 0 & 0 & 0 \\
0 & 0 & 0 & 0 & 0 & 0 & 0 & 0 \\
{{y}_{-1,1}} & i{{y}_{-1,1}} & 0 & 0 & 0 & 0 & 0 & 0 \\
-i{{y}_{-1,1}} & {{y}_{-1,1}} & 0 & 0 & 0 & 0 & 0 & 0 \\
{{y}_{-2,3}} & -{{y}_{-2,4}} & -i{{y}_{-1,2}} & -{{y}_{-1,2}} & 0 & 0 & 0 & 0 \\
{{y}_{-2,4}} & {{y}_{-2,3}} & {{y}_{-1,2}} & -i{{y}_{-1,2}} & 0 & 0 & 0 & 0 \\
{{y}_{-3,5}} & -{{y}_{-3,6}} & {{y}_{-2,3}} & -{{y}_{-2,4}} & -{{y}_{-1,1}} & -i{{y}_{-1,1}} & 0 & 0 \\
{{y}_{-3,6}} & {{y}_{-3,5}} & {{y}_{-2,4}} & {{y}_{-2,3}} & i{{y}_{-1,1}} & -{{y}_{-1,1}} & 0 & 0 \\
\end{pmatrix},\]
where all $y_{a,b}\in \fg_{a}\otimes\mathbb{C}$ are complex.

Since we are working with matrices, we can compute $\phi: \fm\to \fn=\mathbb{C}^6$ directly from the definition without use of the Baker--Campbell--Hausdorff formula. Indeed, in this example $\exp(X)=X^0+X^1+\frac12X^2+\frac16X^3$ holds for all $X\in \fm$, while $\exp^{-1}(Z)=(Z-Z^0)^1$ holds for all $Z\in\exp(\fn)$. In the coordinates given by the above matrices, we get the following formula for $\phi$:
\begin{align*}
{y}_{-1,1}&=\frac{{x}_{-1,1}+i{x}_{-1,2}}{2}\\
{y}_{-1,2}&=\frac{{x}_{-1,4}+i{x}_{-1,3}}{2}\\
{{y}_{-2,3}}&={{x}_{-2,5}}+\frac{i}{2}({x}_{-1,1}{x}_{-1,4}+{x}_{-1,2}{x}_{-1,3})\\
{{y}_{-2,4}}&={{x}_{-2,6}}+\frac{i}{2}({x}_{-1,2}{x}_{-1,4}-{x}_{-1,1}{x}_{-1,3})\\
{y}_{-3,5}&={{x}_{-3,7}}+\frac12({x}_{-1,2}+i{x}_{-1,1})(i{{x}_{-2,5}}+{{x}_{-2,6}})
+\frac1{12}(3i{x}_{-1,2}^2-2{x}_{-1,1}{x}_{-1,2}\\
&-3i{x}_{-1,1}^2){x}_{-1,4}-\frac1{12}({x}_{-1,2}^2+6i{x}_{-1,1}{x}_{-1,2}-{x}_{-1,1}^2){x}_{-1,3}
\\
{{y}_{-3,6}}&={{x}_{-3,8}}-\frac12({x}_{-1,2}+i{x}_{-1,1})({{x}_{-2,5}}-i{{x}_{-2,6}})-\frac1{12}(3i{x}_{-1,2}^2-2{x}_{-1,1}{x}_{-1,2}\\
&-3i{x}_{-1,1}^2){x}_{-1,3}-\frac1{12}({x}_{-1,2}^2+6i{x}_{-1,1}{x}_{-1,2}-{x}_{-1,1}^2){x}_{-1,4}\\
\end{align*}

Next in order to get the defining equations, we eliminate step by step the variables $x_{a,b}$  from the formula for $\phi$. In the first step, we denote $z_1={y}_{-1,1}$, $z_2={y}_{-1,2}$ and we obtain ${x}_{-1,1}=z_1+\bar z_{1},{x}_{-1,1}=-i(z_1-\bar z_{1}),{x}_{-1,3}=-i(z_2-\bar z_{2}),{x}_{-1,4}=z_2+\bar z_{2}$. After elimination of the variable $x_{-1,b}$, the equations of degree two in the formula for $\phi$ become 
\begin{align*}
{{y}_{-2,3}}&={{x}_{-2,5}}+i(z_1\bar z_{2}+z_2\bar z_{1})\\
{{y}_{-2,4}}&={{x}_{-2,6}}+z_1\bar z_{2}-z_2\bar z_{1}\\
\end{align*}
and we see that $\mbox{Re}(w_1)=\mbox{Re}({{y}_{-2,3}})={{x}_{-2,5}}$, $\mbox{Re}(w_2)=\mbox{Re}({{y}_{-2,4}})={{x}_{-2,6}}$  and get the first two defining equations 
\begin{align*}
&\mbox{Im}(w_{1})=z_{1}\bar z_{2}+\bar z_{1}z_{2},\\
&\mbox{Im}(w_{2})=-iz_{1}\bar z_{2}+i\bar z_{1}z_{2}.
\end{align*}
After elimination of the variable $x_{-2,b}$, the equations of degree three in the formula for $\phi$ become 
\begin{align*}
{y}_{-3,5}&={{x}_{-3,7}}-\frac2{3}i(z_1^2\bar z_{2}+2\bar z_{1}^2z_2)+(i\mbox{Re}(w_2)-\mbox{Re}(w_1))\bar z_{1}\\
{y}_{-3,6}&={{x}_{-3,8}}-\frac2{3}(z_1^2\bar z_{2}-2\bar z_{1}^2z_2)-(i\mbox{Re}(w_1)+\mbox{Re}(w_2))\bar z_{1}
\end{align*}
and we see that $\mbox{Re}(w_3)=\mbox{Re}({{y}_{-3,5}})={{x}_{-3,7}}+\frac13i(z_1^2\bar z_{2}-\bar z_{1}^2z_2)+\frac12i(\mbox{Re}(w_{2})+i\mbox{Re}(w_{1}))\bar z_{1}-\frac12(i\mbox{Re}(w_{2})+\mbox{Re}(w_{1}))z_1$ and $\mbox{Re}(w_4)=\mbox{Re}({{y}_{-3,6}})={{x}_{-3,8}}+\frac13(z_1^2\bar z_{2}+\bar z_{1}^2z_2)-\frac12(\mbox{Re}(w_{2})+i\mbox{Re}(w_{1}))\bar z_{1}-\frac12(\mbox{Re}(w_{2})-i\mbox{Re}(w_{1}))z_1$ and get the remaining defining equations
\begin{align*}
&\mbox{Im}(w_{3})=-z_1^2\bar z_{2}-\bar z_{1}^2z_2+\frac12(\mbox{Re}(w_{2})+i\mbox{Re}(w_{1}))\bar z_{1}+\frac12(\mbox{Re}(w_{2})-i\mbox{Re}(w_{1}))z_1,\\
&\mbox{Im}(w_{4})=i(z_1^2\bar z_{2}-\bar z_{1}^2z_2+\frac12(\mbox{Re}(w_{2})+i\mbox{Re}(w_{1}))\bar z_{1}-\frac12(\mbox{Re}(w_{2})-i\mbox{Re}(w_{1}))z_1).
\end{align*}

For the computation of the infinitesimal CR automorphisms of our example, it is convenient to decompose the formula from Section \ref{Sec24} to two parts. First part is the formula
\begin{align*}
&\frac{d}{dt}|_{t=0}g(tZ_{a,b},Y_{c,d})=Z_{-1,1}\frac{\partial}{\partial z_1}+Z_{-1,2}\frac{\partial}{\partial z_2}+Z_{-2,3}\frac{\partial}{\partial w_1}+Z_{-2,4}\frac{\partial}{\partial w_2}\\
&+(Z_{-3,5}-Z_{-1,1}(Y_{-2,3}+iY_{-2,4})+Y_{-1,1}(Z_{-2,3}+iZ_{-2,4}))\frac{\partial}{\partial w_3}\\
&+(Z_{-3,6}-Z_{-1,1}(Y_{-2,4}-iY_{-2,3})+Y_{-1,1}(Z_{-2,4}-iZ_{-2,3}))\frac{\partial}{\partial w_4}
\end{align*}
for the (holomorphic) functions $Z_{a,b},Y_{a,b}$ valued in $\fn\cap (\fg_{a}\otimes \mathbb{C})$.

The second formula $Z_{a,b}=Ad(\exp(Y_{a,b}))(Ad(\exp(-Y_{a,b}))(X))_{\fn}$ for $X\in \frak{sp}(4,\mathbb{C})$ and $Y_{a,b}\in \fn$ is quite complicated and not worth to present here. Nevertheless, it can be easily expressed by multiplication of the matrices and by the projection to $\fn$ and therefore computed in the most of the computer algebras systems (Maple, Maxima,\dots).

\appendix

\section{Further examples}\label{Sec6}

Let us provide some more examples with non--trivial $\fg_+$ that are constructed in the same way as the example in the Section \ref{Sec51}. The first example is interesting, because it provides a realization of the exceptional Lie algebra $\frak{g}_2(\mathbb{C})$ as a Lie algebra of holomorphic vector fields on $\mathbb{C}^{10}$. The second example highlights that our construction also works in the case when the Tanaka prolongation $\fg$ is a simple real Lie algebra. We will not present the full construction of these examples and provide only the minimal sets of infinitesimal CR automorphisms that generate whole $\fg$. These sets consist of a special element $\mathcal{J}$ of $\fg_0$ that is acting as $I$ on $\fg_{-1}$, the real part $\mbox{Re}(\fg_{-l})$ of $\fg_{-l}$ and the real part $\mbox{Re}(\fg_{-1})$ of $\fg_1$.

\subsection{Example $\fg=\frak{g}_2(\mathbb{C}), \Xi=\{1,2\}$}

There is a unique entry in the Table \ref{T1} concerning the exceptional Lie algebra $\frak{g}_2(\mathbb{C})$. The defining equations of the corresponding standard real submanifold in $\mathbb{C}^{10}$ are in fact obtained by adding the following four more equations and variables to the four equations from the example in Section \ref{Sec51}:

\begin{align*}
&\mbox{Im}(w_{1})=z_{1}\bar z_{2}+\bar z_{1}z_{2},\\
&\mbox{Im}(w_{2})=-iz_{1}\bar z_{2}+i\bar z_{1}z_{2},\\
&\mbox{Im}(w_{3})=-z_1^2\bar z_{2}-\bar z_{1}^2z_2+\frac12(\mbox{Re}(w_{2})+i\mbox{Re}(w_{1}))\bar z_{1}+\frac12(\mbox{Re}(w_{2})-i\mbox{Re}(w_{1}))z_1,\\
&\mbox{Im}(w_{4})=i(z_1^2\bar z_{2}-\bar z_{1}^2z_2+\frac12(\mbox{Re}(w_{2})+i\mbox{Re}(w_{1}))\bar z_{1}-\frac12(\mbox{Re}(w_{2})-i\mbox{Re}(w_{1}))z_1),\\
&\mbox{Im}(w_{5})=
\frac14(z_1^3\bar z_{2}+\bar z_{1}^3z_2)
-\frac18(i\mbox{Re}(w_1)+\mbox{Re}(w_2))\bar z_{1}^2
+\frac18(i\mbox{Re}(w_1)-\mbox{Re}(w_2))z_1^2\\
&+\frac38(i\mbox{Re}(w_3)+\mbox{Re}(w_4))\bar z_{1}
+\frac38(-i\mbox{Re}(w_3)+\mbox{Re}(w_4))z_1,\\
&\mbox{Im}(w_{6})=\frac{i}{4}(z_1^3\bar z_{2}-\bar z_{1}^3z_2)+\frac18(i\mbox{Re}(w_2)-\mbox{Re}(w_1))\bar z_{1}^2-\frac18(\mbox{Re}(w_1)+i\mbox{Re}(w_2))z_1^2\\
&+\frac38(-i\mbox{Re}(w_4)+\mbox{Re}(w_3))\bar z_{1}+\frac38(\mbox{Re}(w_3)+i\mbox{Re}(w_4))z_1),\\
&\mbox{Im}(w_{7})=\frac18(z_1^3\bar z_{2}^2+\bar z_{1}^3z_2^2)+\frac{5}{16}(i\mbox{Re}(w_1)-\mbox{Re}(w_2))z_1^2\bar z_{2}-\frac{5}{16}(i\mbox{Re}(w_1)+\mbox{Re}(w_2))\bar z_{1}^2z_2\\
&+\frac{9}{16}(-\mbox{Re}(w_4)+i\mbox{Re}(w_3))z_1\bar z_{2}-\frac{9}{16}(\mbox{Re}(w_4)+i\mbox{Re}(w_3))\bar z_{1}z_2+\frac{1}{4}(i\mbox{Re}(w_6)+\mbox{Re}(w_6))\bar z_{2}\\
&+\frac{1}{4}(-i\mbox{Re}(w_5)+\mbox{Re}(w_6))z_2
+\frac18(-\mbox{Re}(w_1)^2+2i\mbox{Re}(w_2)\mbox{Re}(w_1)+\mbox{Re}(w_2)^2)\bar z_{1}\\
&+\frac18(-\mbox{Re}(w_1)^2-2i\mbox{Re}(w_2)\mbox{Re}(w_1)+\mbox{Re}(w_2)^2)z_1,\\
&\mbox{Im}(w_{8})=
\frac{i}{8}(iz_1^3\bar z_{2}^2-\bar z_{1}^3z_2^2)
-\frac{5}{16}(\mbox{Re}(w_1)+i\mbox{Re}(w_2))z_1^2\bar z_{2}
+\frac{5}{16}(i\mbox{Re}(w_2)-\mbox{Re}(w_1))\bar z_{1}^2z_2\\
&-\frac{9}{16}(\mbox{Re}(w_3)+i\mbox{Re}(w_4))z_1\bar z_{2}
+\frac{9}{16}(-\mbox{Re}(w_3)+i\mbox{Re}(w_4))\bar z_{1}z_2
+\frac14(-\mbox{Re}(w_5)+i\mbox{Re}(w_6))\bar z_{2}\\
&-\frac14(\mbox{Re}(w_5)+i\mbox{Re}(w_6))z_2
+\frac18(i\mbox{Re}(w_1)^2+2\mbox{Re}(w_2)\mbox{Re}(w_1)-i\mbox{Re}(w_2)^2)\bar z_{1}\\
&+\frac18(-i\mbox{Re}(w_1)^2+2\mbox{Re}(w_2)\mbox{Re}(w_1)+i\mbox{Re}(w_2)^2)z_1,
\end{align*}
where $[z_1]=[z_2]=1, [w_1]=[w_2]=2, [w_3]=[w_4]=3,[w_5]=[w_6]=4, [w_7]=[w_8]=5$, i.e., the Bloom--Graham type is $(2,2,3,3,4,4,5,5)$. The Lie algebra of infinitesimal CR automorphisms of this model surface is isomorphic to $\frak{g}_2(\mathbb{C})$ and is $|5|$--graded. Moreover, it is generated by the following weighted homogeneous (for $[\frac{\partial}{\partial z_1}]=[\frac{\partial}{\partial z_2}]=-1,$ $[\frac{\partial}{\partial w_1}]=[\frac{\partial}{\partial w_2}]=-2$, $[\frac{\partial}{\partial w_3}]=[\frac{\partial}{\partial w_4}]=-3$,$[\frac{\partial}{\partial w_5}]=[\frac{\partial}{\partial w_6}]=-4$,$[\frac{\partial}{\partial w_7}]=[\frac{\partial}{\partial w_8}]=-5$) holomorphic vector fields on $\mathbb{C}^{10}$:
\begin{align*}
&\mbox{Re}(\fg_{-5})=\{\frac{\partial}{\partial w_7}\},\\
&\mathcal{J}=iz_1\frac{\partial}{\partial z_1}-iz_2\frac{\partial}{\partial z_2}-2w_2\frac{\partial}{\partial w_1}+2w_1\frac{\partial}{\partial w_2}-3w_4\frac{\partial}{\partial w_3}+3w_3\frac{\partial}{\partial w_4}\\
&+4w_6\frac{\partial}{\partial w_5}-4w_5\frac{\partial}{\partial w_6}+5w_8\frac{\partial}{\partial w_7}-5w_7\frac{\partial}{\partial w_8},\\
&\mbox{Re}(\fg_{1})=\{-2z_1^2\frac{\partial}{\partial z_1}-\frac{3w_2+3iw_1}{2}\frac{\partial}{\partial z_2}-((iw_2+w_1)z_1+2w_3)\frac{\partial}{\partial w_1}\\
&-((w_2-iw_1)z_1+2w_4)\frac{\partial}{\partial w_2}-\frac{(iw_4+w_3)z_1+4w_6}{2}\frac{\partial}{\partial w_3}-\frac{(w_4-iw_3)z_1+4w_5}{2}\frac{\partial}{\partial w_4}\\
&-\frac{(w_4-iw_3)z_1^2-6(iw_6-w_5)z_1}{4}\frac{\partial}{\partial w_5}-\frac{(w_3+iw_4)z_1^2+6(w_6+iw_5)z_1}{4}\frac{\partial}{\partial w_6}\\
&-\frac{i(iw_2+w_1)((iw_4+w_3)z_1+6(w_5-iw_6))}{8}\frac{\partial}{\partial w_7}\\
&+\frac{(iw_2+w_1)((iw_4+w_3)z_1+6(w_5-iw_6))}{8}\frac{\partial}{\partial w_8},\\
&\frac{w_2-iw_1}{2}\frac{\partial}{\partial z_1}-2z_2^2\frac{\partial}{\partial z_2}+(iw_2-w_1)z_2\frac{\partial}{\partial w_1}+i(iw_2-w_1)z_2\frac{\partial}{\partial w_2}\\
&-\frac{i(iw_2-w_1)^2}{2}\frac{\partial}{\partial w_3}+\frac{(iw_2-w_1)^2}{2}\frac{\partial}{\partial w_4}\\
&+\frac{4(iw_6+w_5)z_2-8w_7+3(iw_2-w_1)(iw_4-w_3)}{8}\frac{\partial}{\partial w_5}\\
&+\frac{4(w_6-iw_5)z_2-8w_8-3i(iw_2-w_1)(iw_4-w_3)}{8}\frac{\partial}{\partial w_6}\\
&-\frac{4(iw_6+w_5)z_2^2+4(iw_8+w_7)z_2+3(iw_2-w_1)(iw_4-w_3)z_2+i(iw_2-w_1)^3}{8}\frac{\partial}{\partial w_7}\\
&-\frac{4(w_6-iw_5)z_2^2+4(w_8-iw_7)z_2-3i(iw_2-w_1)(iw_4-w_3)z_2+(iw_2-w_1)^3}{8}\frac{\partial}{\partial w_8}\}.
\end{align*}

\subsection{Example $\fg=\frak{su}(2,3), \Xi=\{1,2,3,4\}$}

In the classification in \cite{MS}, we can find the case of simple real Lie algebra $\fg$ with $l>2$ of the smallest rank, which is $\fg=\frak{su}(2,3), \Xi=\{1,2,3,4\}$. We can compute the defining equations of the corresponding standard real submanifold in $\mathbb{C}^8$:

\begin{align*}
\mbox{Im}(w_1)&=iz_1\bar z_{2}-iz_2\bar z_{1},\
\mbox{Im}(w_2)=-z_1\bar z_{2}-z_2\bar z_{1},\
\mbox{Im}(w_3)=-2z_2\bar z_{2},\\
\mbox{Im}(w_4)&=(z_1+\bar z_{1})z_2\bar z_{2}+\frac{1}{4}(\mbox{Re}(w_2)+i\mbox{Re}(w_1))z_2+\frac{1}{4}(\mbox{Re}(w_2)-i\mbox{Re}(w_1))\bar z_{2}\\
&-\frac{1}{4}i(z_1-\bar z_{1})\mbox{Re}(w_3),\\
\mbox{Im}(w_5)&=-i(z_1-\bar z_{1})z_2\bar z_{2}+\frac{1}{4}(\mbox{Re}(w_2)+i\mbox{Re}(w_1))z_2+\frac{1}{4}(\mbox{Re}(w_2)-i\mbox{Re}(w_1))\bar z_{2}\\
&-\frac{1}{4}(z_1+\bar z_{1})\mbox{Re}(w_3),\\
\mbox{Im}(w_6)&=-\frac{2}{3}(z_2\bar z_{2}+z_2+\bar z_{2})z_1\bar z_{1}-\frac{1}{2}(\mbox{Re}(w_5)-i\mbox{Re}(w_4))\bar z_{1}-\frac{1}{2}(\mbox{Re}(w_5)+i\mbox{Re}(w_4))z_1\\
&+\frac{1}{8}(\mbox{Re}(w_2)+i\mbox{Re}(w_1)-i\mbox{Re}(w_2)+\mbox{Re}(w_1))z_1z_2\\
&+\frac{1}{8}(\mbox{Re}(w_2)-i\mbox{Re}(w_1)+i\mbox{Re}(w_2)+\mbox{Re}(w_1))\bar z_{1}\bar z_{2}\\
&+\frac{1}{8}(3i\mbox{Re}(w_2)-3\mbox{Re}(w_1)-\mbox{Re}(w_2)-i\mbox{Re}(w_1))\bar z_{1}z_2\\
&+\frac{1}{8}(-3i\mbox{Re}(w_2)-3\mbox{Re}(w_1)-\mbox{Re}(w_2)+i\mbox{Re}(w_1))z_1\bar z_{2},
\end{align*}
where $[z_1]=[z_2]=1, [w_1]=[w_2]=[w_3]=2, [w_4]=[w_5]=3,[w_6]=4$, i.e., the Bloom--Graham type is $(2,2,2,3,3,4)$. The Lie algebra of infinitesimal CR automorphisms of this model surface is isomorphic to $\frak{su}(2,3)$ and is $|4|$--graded. Moreover, it is generated by the following weighted homogeneous (for $[\frac{\partial}{\partial z_1}]=[\frac{\partial}{\partial z_2}]=-1,$ $[\frac{\partial}{\partial w_1}]=[\frac{\partial}{\partial w_2}]=[\frac{\partial}{\partial w_3}]=-2$, $[\frac{\partial}{\partial w_4}]=[\frac{\partial}{\partial w_5}]=-3$,$[\frac{\partial}{\partial w_6}]=-4$) holomorphic vector fields on $\mathbb{C}^{8}$:
\begin{align*}
\mbox{Re}(\fg_{-4})&=\fg_{-4}=\{\frac{\partial}{\partial w_6}\},\\
\mathcal{J}&=iz_1\frac{\partial}{\partial z_1}-iz_2\frac{\partial}{\partial z_2}+2w_2\frac{\partial}{\partial w_1}-2w_1\frac{\partial}{\partial w_2}-w_5\frac{\partial}{\partial w_4}+w_4\frac{\partial}{\partial w_5},\\
\mbox{Re}(\fg_{1})&=\{-2z_1^2\frac{\partial}{\partial z_1}-\frac{iw_2+w_1}{2}\frac{\partial}{\partial z_2}+z_1(iw_2-w_1)\frac{\partial}{\partial w_1}-z_1(w_2+iw_1)\frac{\partial}{\partial w_2}\\
&-(2w_4-z_1w_3-z_2(w_2+iw_1))\frac{\partial}{\partial w_3}-(w_6+iz_1w_5+\frac{z_1^2w_3+z_1z_2(w_2+iw_1)}{2}\\
&-\frac{i(w_1^2+w_2^2)}{4})\frac{\partial}{\partial w_4}+(iz_1w_4+\frac{iz_1^2w_3+z_1z_2(iw_2-w_1)}{2}-\frac{w_1^2+w_2^2}{4})\frac{\partial}{\partial w_5}\\
&-(z_1w_6+z_1^2(w_4-iw_5)+iz_1\frac{w_1^2+w_2^2}{2})\frac{\partial}{\partial w_6},\\
&\frac{w_1-iw_2}{2}\frac{\partial}{\partial z_1}+\frac{iw_3-4z_2^2}{2}\frac{\partial}{\partial z_2}+\frac{2w_5-iz_1w_3-z_2(3iw_2+w_1)}{2}\frac{\partial}{\partial w_1}\\
&+\frac{2w_4+z_1w_3-z_2(w_2-3iw_1)}{2}\frac{\partial}{\partial w_2}-2z_2w_3\frac{\partial}{\partial w_3}\\
&-\frac{z_2(iw_5+w_4)+(iw_2+w_1-z_1z_2)w_3-z_2^2(w_2+iw_1)}{2}\frac{\partial}{\partial w_4}\\
&-\frac{z_2(w_5-iw_4)+(-w_2+iw_1+iz_1z_2)w_3+z_2^2(iw_2-w_1)}{2}\frac{\partial}{\partial w_5}\\
&+\frac{((w_2-iw_1)w_5-(iw_2+w_1)w_4+z_1(iw_2+w_1)w_3+iz_2(w_1^2+w_2^2)}{2}\frac{\partial}{\partial w_6}\}.
\end{align*}

\end{document}